\documentclass{l4dc2026}
\usepackage{soul}
\usepackage{mathtools}
\usepackage{mathrsfs,dsfont}
\usepackage{enumerate}
\usepackage{bookmark}
\usepackage{tikz}
\usetikzlibrary{arrows.meta,calc,backgrounds}
\usepackage{caption}
\newcommand{\R}{{\mathbb{R}}}
\newcommand{\Z}{{\mathbb{Z}}}
\newcommand{\cF}{\mathcal{F}}

\newcommand{\cR}{\mathcal{R}}
\newcommand{\cX}{\mathcal{X}}
\newcommand{\sA}{\mathsf{A}}

\newcommand{\sP}{\mathsf{P}}

\newcommand{\mvmap}{{\rightrightarrows}}
\newcommand{\setof}[1]{\left\{ {#1}\right\}}
\newcommand{\setdef}[2]{\left\{ #1 \, \middle| \, #2 \right\}}
\newcommand{\Int}{\mathop{\mathrm{int}}\nolimits}
\newcommand{\cl}{\mathop{\mathrm{cl}}\nolimits}
\newcommand{\Inv}{\mathop{\mathrm{Inv}}\nolimits}
\newcommand{\diam}{\mathop{\mathrm{diam}}\nolimits}
\newcommand{\MG}{\mathop{\mathsf{MG}}\nolimits}
\newcommand{\RoA}{\mathop{\mathrm{RoA}}\nolimits}

\newtheorem{thm}{Theorem}[section]

\newtheorem{prop}[thm]{Proposition}

\newtheorem{defn}[thm]{Definition}

\newtheorem{rem}[thm]{Remark}

\title[Topological Dynamics via Learned Hybrid Systems]{Topological Dynamics via Learned Hybrid Systems}

\usepackage{times}
\author{
  \Name{Bernardo Rivas} \Email{bernardo.dopradorivas@utoledo.edu}\\
  \Name{William Kalies} \Email{william.kalies@utoledo.edu}\\
  \addr University of Toledo, Toledo, OH
  \AND
  \Name{Kaito Iwasaki} \Email{kaitoi@umich.edu}\\
  \Name{Anthony Bloch} \Email{abloch@umich.edu}\\
  \Name{Maani Ghaffari} \Email{maanigj@umich.edu}\\
  \addr University of Michigan, Ann Arbor, MI
}

\begin{document}

\maketitle

\begin{abstract}
  The analysis of global dynamics, particularly the identification and characterization of attractors and their regions of attraction, is essential for complex nonlinear and hybrid systems. Combinatorial methods based on Conley's index theory have provided a rigorous framework for this analysis. However, the computation relies on rigorous outer approximations of the dynamics over a discretized state space, which is challenging to obtain from scattered trajectory data. We propose a methodology that integrates recent advances in switching system identification via convex optimization to bridge this gap between data and topological analysis. We leverage the identified switching system to construct combinatorial outer approximations. This paper outlines the integration of these methods and evaluates the efficacy of computing Morse graphs versus data-driven and statistical approaches.
\end{abstract}


\section{Introduction}
\label{sec:introduction}

Topological dynamics provides a topological characterization of attractors of a dynamical system. For a system of ordinary differential equations, $\dot{x} = f(x)$, this involves identifying invariant sets, their stability, and the connectivity between them. Computational topological methods have successfully provided tools to rigorously identify and characterize even chaotic dynamics \citep{Mischaikow1995Chaos,ComputationalHomology}. Recently, these tools have been applied towards controller design and safety evaluation in robotics due to their superior ability in characterizing regions of attraction (RoAs) against Lyapunov-based methods \citep{Ewerton2022MG,Ewerton2022GP,MORALS}.

Fundamentally, the method consists of the computation of a \emph{Morse graph}, whose existence is rooted in Conley index theory. The computation relies on a combinatorial representation of the dynamics over a discretization of the state space $\cX$. The main bottleneck is the construction of a multivalued map $\cF\colon\cX \mvmap \cX$ that serves as an \emph{outer approximation} of the underlying continuous flow.

The primary limitation in applying these topological methods to data-driven scenarios is the difficulty of guaranteeing the outer approximation property from trajectory data. \citet{Bogdan2020,Yim25} have attempted to address the issue with respect to Conley index theory, nonetheless, errors in this approximation can lead to spurious topological features or missed dynamics. \citet{Bogdan2024} and \citet{Ewerton2022GP} have employed a probabilistic characterization of the dynamics, constructing outer approximations with confidence levels rather than deterministic guarantees that allows the reconstruction of topological invariants.

We propose to address this limitation by utilizing identified dynamical models obtained via the convex optimization framework for switching system identification as in \citet{Kaito2025}. This framework identifies Switching Linear Systems (SLSs) and Switching Polynomial Systems (SPSs) from data, yielding a piecewise-smooth vector field $\hat{f}$. The analytical structure of the identified model enables the construction of outer approximations $\cF$, therefore enabling the combinatorial representation of the dynamics. If the learned system $\hat{f}$ approximates the true dynamics $f$ with sufficiently small error, the topological information computed for $\hat{f}$ is equivalent to that of $f$ \citep{ContinuationSheaves}.

This paper outlines the integration of convex optimization-based system identification and combinatorial dynamics. We detail how the output of switching system identification enables the computation of Morse graphs and compare the resulting regions of attraction against alternative state-of-the-art, data-driven and statistical methods.

\section{Topological Dynamics}
\label{sec:dynamics}

We consider a dynamical system $\varphi : \R \times X \to X$ on a compact region $X \subset \R^n$ defined by solutions of $\dot{x} = f(x)$ where $f: X \to \R^n$. An \emph{attractor} is an invariant set $A$ for which there exists an open neighborhood $U$ such that $A=\omega(U)\subset U$, where $\omega$ is the omega-limit.  
The \emph{region of attraction} (RoA) of an attractor $A$ is the set of all initial conditions whose forward trajectories converge to $A$.

Conley index theory offers an algebraic topological framework that rigorously characterizes such attractors as special cases of isolated invariant sets. A compact set $S \subseteq X$ is \emph{invariant} if $\varphi(t,S) = S$ for all $t \in \R$. A compact set $N$ is an \emph{attracting block} if $\varphi(t,N) \subset \Int(N)$ for all $t > 0$. Given nested attracting blocks $N_0 \subset N_1 \subset X$, 
we may study the corresponding Morse set, defined as the maximal invariant set contained in the interior of their set difference: $S \coloneqq \Inv(\cl(N_1 \setminus N_0))$. 

The \emph{(homological) Conley index} of $S$ is the relative homology of the pair $(N_1,N_0)$, namely $CH_*(S) = H_*(N_1,N_0)$. This algebraic topological invariant is independent of the choice of attracting blocks and depends only on the isolated invariant set $S$ \citep{Handbook2002}. The Conley index $CH_*(S)$ provides information about the dynamical behavior within $S$. For instance, if $CH_k(S) \cong \Z$ for some $k \geq 0$ and $CH_n(S) = 0$ for $n\neq k$, then $S$ contains at least one fixed point \citep{McCord1988}. Similar results allow identification of periodic orbits \citep{McCord1995} or chaotic behavior \citep{Mischaikow1995Chaos}. Furthermore, the theory prescribes the existence of a \emph{Morse decomposition} that decomposes the space into finitely many Morse sets that are partially ordered by the flow and contain all the (chain-)recurrent behavior, so that the system is gradient-like outside of the Morse sets.

A key property of topological dynamics is \emph{robustness with respect to perturbations}: attracting blocks, Morse decompositions, and associated topological invariants are preserved under sufficiently small perturbations of the flow \citep{Conley71,ContinuationSheaves}. This robustness justifies the use of an identified approximation $\hat{f} \approx f$ for topological analysis, provided the approximation error is sufficiently small.

\section{Switching System Identification}
\label{sec:optimization}

Switching systems constitute a class of hybrid dynamical systems that evolve according to one of several constituent vector fields, with a discrete mode selecting which dynamics are active at a given state. Formally, a switching system has the form
\[
  \dot{x} = f_{\sigma(x)}(x), \quad \sigma : X \to \setof{1,\ldots,K},
\]
where each mode $i \in \setof{1,\ldots,K}$ is associated with a continuous vector field $f_i: X \to \R^n$. Although the overall system may exhibit discontinuities at the mode boundaries, the dynamics are continuous within each mode. This can be expressed compactly using binary indicator variables as $f(x) = \sum_{j=1}^K \lambda_j(x) f_j(x)$ where $\lambda_j(x) \in \setof{0,1}$ and $\sum_{j=1}^K \lambda_j(x) = 1$.

Recent work by \citet{Kaito2025} has developed a convex optimization framework for identifying switching systems from trajectory data. Given measurements $\{(x^i, \dot{x}^i)\}_{i=1}^N$, the identification problem is formulated as minimizing the discrepancy
\[
  \min_{\lambda^i, C_j} \sum_{i=1}^N \left\| \dot{x}^i - \sum_{j=1}^K \lambda_j^i C_j \Phi(x^i) \right\|_p
\]
where $\lambda_j^i = \lambda_j(x^i)$ are subject to mode constraints $\lambda_j^i \in \setof{0,1}$ and $\sum_j \lambda_j^i = 1$ while $\Phi(x)$ is a feature map (e.g., $\Phi(x) = x$ for SLS or $\Phi(x) = \phi_d(x)$, some polynomial basis up to degree-$d$ polynomial for SPS) and $C_j$ are coefficients matrices (e.g., $C_j = A_j$ for SLS or a matrix of polynomial coefficients for SPS).

The above mixed-integer formulation struggles with  certifiability due to the inherent nonlinearity arising from the cross terms between decision variables $\lambda_j^i$ and $C_j$, and the computational cost due to the mixed-integer constraints. We handle both challenges using bilevel convex optimizations that alternates between mode assignment via semidefinite programming (SDP) using the moment relaxation, and dynamics estimation via linear programming (LP). The key is to realize the mixed-integer constraints on $\lambda_j^i$ as polynomial constraints
\[\lambda_j^i(\lambda_j^i - 1) = 0, \quad \mathbf{1}^T\lambda^i = 1,\]
so that we can relax them to the order-$r$ moment relaxation
\[M_r(y) \succeq 0, \quad M_{r-1}(\lambda_j^i(\lambda_j^i-1) y)=0, \quad M_{r-1}((\mathbf{1}^T \lambda^i-1) y)=0,\]
where $y = (y_\alpha)_{|\alpha| \le 2r}$ is a truncated moment sequence indexed by monomials in $\lambda$, $M_r(y)$ is the moment matrix, and $M_{r-d}(g y)$ is the localizing matrix associated with constraint polynomial $g$ with $d = \lceil \operatorname{deg}g/2\rceil$. More details of the moment relaxation in polynomial optimization can be found in \cite{Lasserre2015}. We choose the $\ell_1$-norm objective and solve the bilevel convex optimizations to identify $\lambda_j^i$ and $C_j$. More specifically, switching system identification using order-1 moment relaxation can be formulated as follows. By using slack variables $\delta_k^i \ge 0$ for $k=1,\dots,n$, and denoting $c_k^i(C,\lambda) = [\dot z^{i} - \sum_{j=1}^M \lambda_j^i C_j \Phi(z^i)]_k$, we solve the following SDP using by fixing the coefficients $C_j$ of the vector fields for $\lambda_j^i$:
\[
  \begin{aligned}
    \min_{\lambda^i,\Lambda^i,\delta^i}\quad & \sum_{i=1}^N\sum_{k=1}^n \delta_k^i \quad \text{s.t.}\   -\delta_k^i \le c_k^i(C,\lambda) \le \delta^i_k, \quad \delta^i_k \ge 0, \\
    &
    \begin{bmatrix}
      1 & (\lambda^i)^T \\
      \lambda^i & \Lambda^i
    \end{bmatrix} \succeq 0, \quad \operatorname{diag}(\Lambda^i) = \lambda^i,\quad \mathbf{1}^T \lambda^i = 1,
  \end{aligned}
\]
where we use $\lambda$ to denote the first order moments $y_j$, and $\Lambda_{ij} = y_{ij}$. Modes are recovered either directly from $\lambda^i$ or by hardening to $e_{j^\star}$ with $j^\star\in\arg\max_j \lambda_j^i$. At the initial iteration, we choose  arbitrary coefficients for the fixed $C_j$. For subsequent iterations, $C_j$ come from solving the following dynamics search LP. Note that for higher-order relaxations $(r>1)$, the linear inequalities become LMIs on localizing matrices. With these fixed $\lambda^i$, we solve the following LP for $C_j$:
\[
  \min_{C_j,\delta^i}\  \sum_{i=1}^N \sum_{k=1}^n \delta_k^i \quad
  \text{s.t.}\ -\delta_k^i \le c_k^i, \le \delta^i_k,\ \delta^i_k \ge 0, \ -\eta \le (C_j)_{\ell m} \le \eta.
\]
We alternate these convex subproblems to obtain the estimates $\hat{\lambda}^i$ and $\hat{C}_j$.
\begin{remark}
  In practice, warm-starting the alternating convex optimization improves convergence speed and accuracy, particularly for systems with a higher number of modes or with overlapping domains. For example, mode assignments $\lambda^i$ are initialized by clustering velocity data (e.g., K-means on $\dot{x}^i$) followed by cluster-wise least-squares regression to initialize $C_j$.
\end{remark}
Once we obtain these mode labels $\hat{\lambda}^i$ for the $N$ samples $(x^i, \dot{x}^i)$, a classical soft-margin polynomial classifier method is used to reconstruct the state-dependent switching law $\hat{\lambda}: \mathbb{R}^n \to \setof{0,1}^M$ with $\mathbf{1}^T \hat{\lambda}(x) = 1$. Consequently, the output is an identified system
\[
  \hat{f}(x) = \sum_{j=1}^M \hat{\lambda}_j(x) \hat{f}_j(x),
\]
where $\hat{f}_j(x) = \hat{A}_j x$ for SLS or $\hat{f}_j(x) = \hat{C}_j \phi_d(x)$ for SPS.

\section{Combinatorial Dynamics}
\label{sec:combinatorial}

Our goal is to analyze the discrete-time dynamics induced by the time-$\tau$ map $\varphi_\tau: X \to X$ for some fixed $\tau > 0$ and obtain insights into the dynamical structure of $\varphi$ as in \cite{LatticeI,LatticeII,LatticeIII}. Given the identified switching system $\hat{f}$, we construct a combinatorial representation of its dynamics over a discretization of the state space as in \citet{Kalies2005}.

\begin{defn}
  A grid $\cX$ on $X \subset \R^n$ is a finite set together with a geometric realization map $|\cdot| : \cX \to X$ such that for each $\xi \in \cX$, $|\xi|$ is homeomorphic to a closed ball in $\R^n$, $X = \bigcup_{\xi \in \cX} |\xi|$, and $|\xi| \cap \Int(|\xi'|) = \emptyset$ for all $\xi \neq \xi' \in \cX$. The \emph{diameter} of $\xi$ is $\diam(\xi) = \max_{x,x' \in |\xi|} \| x-x' \|$, and the diameter of $\cX$ is $\diam(\cX)=\max_{\xi \in \cX} \diam(\xi)$.
\end{defn}
\begin{rem}
  In practice, we use cubical grids as in \citet{ComputationalHomology}. For a cubical cell $\xi \in \cX$, we denote by $V(\xi)$ the set of $2^n$ vertices (corner points) of $|\xi|$. Note that for any $x \in |\xi|$, there exists $v \in V(\xi)$ such that $\|x - v\| \leq \diam(\xi)/2$.
\end{rem}

\subsection{Outer Approximation}

\begin{defn}
  A multivalued map $\cF: \cX \mvmap \cX$ is an outer approximation of $\hat{\varphi}_\tau$ if
  \[
    \hat{\varphi}_\tau(|\xi|) \subset \Int \left| \cF(\xi) \right|, \quad \forall \xi \in \cX.
  \]
\end{defn}
The key advantage of identifying an analytical switching system $\hat{f}$ from data is that it induces a flow $\hat{\varphi}$ whose time-$\tau$ map $\hat{\varphi}_\tau$ can be systematically approximated. For each cell $\xi \in \cX$, we integrate the system from the cell's corners (and optionally its center) for time $\tau$ to obtain points $\setdef{\hat{\varphi}_\tau(v)}{v \in V(\xi)}$, where $V(\xi)$ denotes the vertices of $|\xi|$. We then compute a bounding box $\cR_\tau(|\xi|)$ enclosing these integrated points and define
\[
  \cF(\xi) := \setdef{ \xi' \in \cX }{|\xi'| \cap \cR_\tau(|\xi|) \neq \emptyset }.
\]

For a fine enough grid, this bounding box construction provides an outer approximation by exploiting the analytical structure of the identified model. Alternatively, Lipschitz bounds can be computed analytically for each mode as in \citet{Ewerton2022MG}, allowing for an outer approximation at a given resolution, though the quality of the outer approximation depends on the Lipschitz constants.

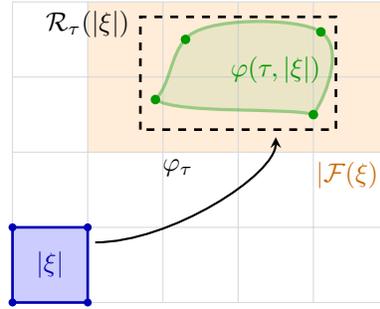
\begin{figure}[ht]
  \centering
  \begin{tikzpicture}[
      scale=1.0,
      >=stealth,
      every node/.style={font=\small},
    ]

    \colorlet{gridcolor}{gray!30}
    \colorlet{source_color}{blue!70!black}
    \colorlet{source_fill}{blue!20}
    \colorlet{target_color}{green!60!black}
    \colorlet{target_fill}{green!60!black}
    \colorlet{outer_approx_color}{orange!80!black}
    \colorlet{outer_approx_fill}{orange!15}
    \colorlet{bounding_box_color}{black}

    \draw[gridcolor, line width=0.4pt] (0,0) grid[step=1] (5,4);
    \fill[source_fill] (0,0) rectangle (1,1);
    \draw[source_color, line width=1pt] (0,0) rectangle (1,1);
    \node[source_color] at (0.5,0.5) {$|\xi|$};

    \foreach \x/\y/\name in {0/0/1, 1/0/2, 1/1/3, 0/1/4} {
      \fill[source_color] (\x,\y) circle (1.5pt);
      \coordinate (v\name) at (\x,\y);
    }

    \begin{scope}[on background layer]
      \fill[outer_approx_fill] (1,2) rectangle (5,4);
    \end{scope}
    \coordinate (w1) at (1.9,2.7);
    \coordinate (w2) at (4.0,2.5);
    \coordinate (w3) at (4.1,3.6);
    \coordinate (w4) at (2.3,3.5);

    \draw[->, black, line width=0.8pt]
    (1.1,0.8) .. controls ++(1.0,0.0) and ++(-0.0,-0.5) .. (3.5,2.2);
    \foreach \i in {1,2,3,4} {
      \fill[target_color] (w\i) circle (1.8pt);
    }

    \draw[bounding_box_color, dashed, line width=1pt]
    (1.7,2.3) rectangle (4.3,3.8);
    \draw[target_color, line width=1.2pt, opacity=0.4, fill=target_fill, fill opacity=0.08]
    (w1) .. controls ++(0.3,-0.2) and ++(-0.3,0.1) .. (w2)
    .. controls ++(0.2,0.3) and ++(0.3,-0.2) .. (w3)
    .. controls ++(-0.3,0.2) and ++(0.2,0.3) .. (w4)
    .. controls ++(-0.2,-0.2) and ++(0.2,0.2) .. (w1);
    \node[black] at (2.2,1.8) {$\varphi_\tau$};
    \node[bounding_box_color, font=\small] at (1.0,3.7) {$\cR_\tau(|\xi|)$};
    \node[outer_approx_color] at (4.5,1.7) {$|\cF(\xi)|$};
    \node[target_color, font=\small] at (3.5,3.1) {$\varphi(\tau,|\xi|)$};

  \end{tikzpicture}%
  \caption{Outer approximation of time-$\tau$ map via the bounding box $\cR_\tau(|\xi|)$.}
  \label{fig:outer_approx_schematic}
\end{figure}

When the identified flow $\hat{\varphi}_\tau$ approximates the true flow $\varphi_\tau$ with known error bounds, we can leverage this approximation to construct an outer approximation of the true dynamics.

\begin{prop}\label{thm:outer_approximation}
  If $\|\hat{\varphi}_\tau - \varphi_\tau\| \leq \delta$ for some $0<\delta< \diam(\cX)/2$, then the coarser outer approximation of $\hat{\varphi}_\tau$
  \[
    \cF_c(\xi) = \setdef{\xi' \in \cX}{|\xi'| \cap B(\cR_\tau(|\xi|), \delta) \neq \emptyset}
  \]
  is an outer approximation of $\varphi_\tau$. Furthermore, if $\delta < d(\varphi_\tau(|\xi|), \partial|\cF(\xi)|)$ where
  \[
    \cF(\xi) = \setdef{\xi' \in \cX}{|\xi'| \cap \cR_\tau(|\xi|) \neq \emptyset},
  \]
  then $\cF_c = \cF$.
\end{prop}

\begin{proof}
  We have $\varphi_\tau(|\xi|) \subset B(\hat{\varphi}_\tau(|\xi|), \delta)$ by assumption, and $\hat{\varphi}_\tau(|\xi|) \subseteq \cR_\tau(|\xi|)$ by construction. Thus $\varphi_\tau(|\xi|) \subset B(\cR_\tau(|\xi|), \delta)$. By definition of $\cF_c(\xi)$, any cell intersecting $B(\cR_\tau(|\xi|), \delta)$ is included, so $\varphi_\tau(|\xi|) \subset \Int|\cF_c(\xi)|$. Furthermore, notice that $\cF(\xi) \subseteq \cF_c(\xi)$ by construction. If $\delta < d(\varphi_\tau(|\xi|), \partial|\cF(\xi)|)$, then $B(\cR_\tau(|\xi|), \delta) \subset |\cF(\xi)|$, so $\cF_c(\xi)=\cF(\xi)$ for any $\xi \in \cX$.
\end{proof}

\subsection{Morse Graph and Regions of Attraction}

Morse graphs play the role of Morse decompositions in combinatorial dynamics. They identify a recurrent-like and a gradient-like structure as follows. 
The combinatorial map $\cF$ induces a directed graph with vertices $\cX$ and edges $(\xi,\xi') \in E$ when $\xi' \in \cF(\xi)$. The strongly connected components (SCCs) partition $\cX$ into trivial and nontrivial SCCs. The \emph{condensation graph} $CG(\cF)$ is the directed acyclic graph whose vertices are the SCCs. The reachability relation in $CG(\cF)$ induces a partial order $\leq$ on its vertices. Let $(\sP,\leq)$ be a poset that indexes the nontrivial SCCs where the partial order is inherited from reachability in $CG(\cF)$. We denote by $M(p)$ the set of elements in the equivalence class and call them Morse sets.
\begin{defn}
  The Morse Graph $\MG(\cF)$ is the Hasse diagram of the poset $(\sP, \leq)$.
\end{defn}
For each $p\in \sP$, we associate an attractor via $\sA(p)=\omega(|\bigcup_n \cF^n(M(p))|)$ \citep{LatticeII}. 
Its \emph{region of attraction} is defined by the reachability of $\cF$:
\[	
    \RoA(p) = \left|\setdef{\xi \in \cX}{ \exists k \geq 0 \text{ s.t. } \cF^k(\xi) \subseteq \bigcup_{n \geq 0} \cF^n(M(p)) }\right|.  
\]

\section{Examples}
\label{sec:examples}
We compare four approaches for constructing outer approximations and computing Morse sets, Morse graphs and RoAs:
\begin{enumerate}
  \item \textbf{Ground-truth:} When the true dynamics $f$ are known analytically, we construct the outer approximation $\cF$ via numerical integration of $f$ using the bounding box method described in Section~\ref{sec:combinatorial} \citep{Kalies2005}. This serves as the reference outer approximation.
  \item \textbf{Lipschitz:} Following the approach in \citet{Ewerton2022MG}, we construct the outer approximation using a global Lipschitz constant $L_\tau$ estimated using $f$:
    \[
      \cF_{L}(\xi) = \setdef{ \xi' \in \cX }{ |\xi'| \cap B(\varphi_\tau(v), L_\tau d/2) \neq \emptyset \text{ for some } v \in V(\xi) },
    \]
    where $d=\diam(\cX)$. 
  \item \textbf{Gaussian-Process:} Following \citet{Ewerton2022GP}, we learn the flow from trajectory data using Gaussian-Processes and compute a probabilistic-inspired outer approximation
    \[
      \cF_{GP}(\xi) =\setdef{ \xi' \in \cX}{|\xi'| \cap \left( E^\delta_{\Sigma(\xi)} \cup \cR_\tau(|\xi|) \right) \neq \emptyset},
    \]
    where $E^\delta_{\Sigma(\xi)} \coloneqq \prod_{n=1}^{M} I^\alpha_{n,\Sigma(\xi)}$ is the confidence ellipsoid from GP predictions at the center of $|\xi|$ with confidence level $1-\delta=0.95$.
  \item \textbf{Ours:} Using the same trajectory data used by the Gaussian-Process approach, together with the vector field on the trajectory, we identify a switching system $\hat{f}$ as in Section~\ref{sec:optimization} and construct the outer approximation $\hat{\cF}$ as in Section~\ref{sec:combinatorial}.
\end{enumerate}
For each example, we compare the obtained Morse Graphs. For a fair comparison of the Morse sets and RoAs, we define an \emph{Aggregated Morse Graph} 
from the the expected Morse decomposition. Then, for each method, we match all computed Morse sets to the Aggregated Morse Graph. 
This matching is done by spatial clustering while respecting the partial order. 
Then we compute the Intersection-over-Union (IoU) of Aggregated Morse sets and Aggregated Regions of Attraction with respect to the reference method.

\subsection{Toggle Switch: Bistability in Gene Regulation}

We consider a genetic toggle switch  where two genes repress each other to create bistability \citep{Gardner2000}. We may represent it by a planar piecewise linear system with state $x = (x_1, x_2) \in \R^2$ representing protein concentrations:
\[
  \dot{x} = A x + b_{ij}, \quad x \in \R^2, (i,j) \in \setof{0,1}^2
\]
where $A = -\mathrm{diag}(\gamma_1,\gamma_2)$
and the production vector $b_{ij}$ switches at thresholds $T_1, T_2 > 0$:
\[
  b_{ij} =
  \begin{bmatrix}
    L_1 \\ L_2
  \end{bmatrix} +
  \begin{bmatrix}
    H(T_2-x_2)(U_1 - L_1) \\
    H(T_1-x_1)(U_2 - L_2)
  \end{bmatrix}, \text{ where $H$ is the Heaviside function }
  H(x) =
  \begin{cases}
    0 & x < 0, \\
    1 & x \geq 0.
  \end{cases}
\]
The parameters $L_i,U_i$ represent the basal and activated production rates. \citet{RookFields} have shown that, for gene regulatory networks modeled by sufficiently steep functions, the dynamics do not rely on the precise parameters, but on a set of inequalities that must be satisfied. In this setting, bistability is a byproduct of the simple assumption $0 < L_i < \gamma_i T_i < U_i$. We set the parameters to $\gamma_i = 1.0$, $T_i = 3.0$, $L_i = 1.0$ and $U_i = 5.0$ for $i=1,2$.

The system has two stable fixed points corresponding to states where one gene dominates and represses the other. The separatrix between these basins of attraction passes through an unstable saddle-like region. Note, however, that for the switching system, the saddle point is not well-defined. These three regions emerge as distinct Morse sets in the combinatorial analysis.

We generate $N=50$ trajectories over $[0,10]$ with $\tau = 1.0$ and discretize $[0,6]^2$ into $2^7 \times 2^7$ cubical cells. Figure~\ref{fig:toggleswitch} shows the computation of Morse sets and graphs derived from each outer approximation. Figure~\ref{fig:toggle_combined} shows spatially aggregated results while respecting its connectivity. Table~\ref{tab:toggle_metrics} summarizes the quantitative comparison. 

Our method learns a switching linear system $\hat{f}(x) = Ax + b_{\sigma(x)}$ where 
$A = -I_2$ and the mode $\sigma(x) \in \{1,2,3,4\}$ is determined by learned linear switching surfaces 
$s_i(x) = c_i^T x + d_i$ via $\sigma(x) = 1 +\mathds{1}_{s_1(x) > 0} + 2\cdot \mathds{1}_{s_2(x) > 0}$, 
with production vectors $b_0 = [1,5]^T, b_1 = [1,1]^T, b_2 = [5,5]^T, b_3 = [5,1]^T$. Due to the relatively small sample size, the recovered switching surfaces exhibit overfitting, which causes imperfect alignment with the ground-truth surfaces, despite the perfect matching of the mode assignments in the SDP step (see Figure~\ref{fig:learned sls}). Nevertheless, we still obtain an isomorphic Morse Graph to the ground-truth as expected from Proposition \ref{thm:outer_approximation}.

\begin{figure}[ht]
    \centering
    \footnotesize
    \begin{tabular}{cc}
    \includegraphics[width=0.25\textwidth]{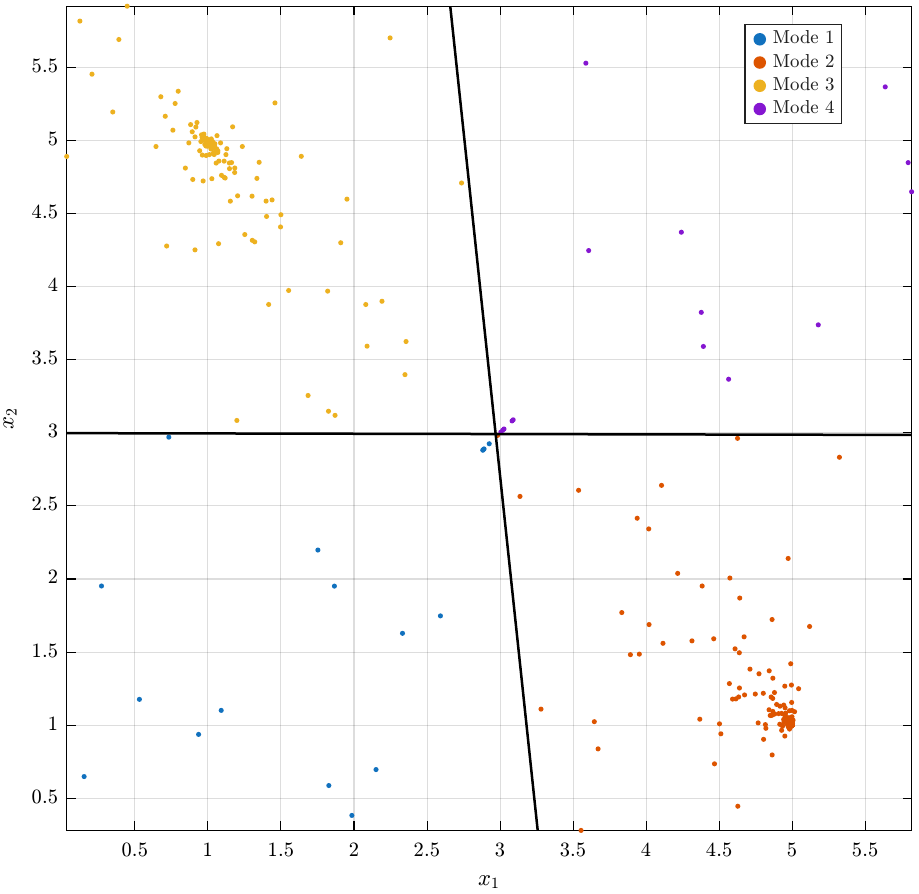}&
    \includegraphics[width=0.25\textwidth]{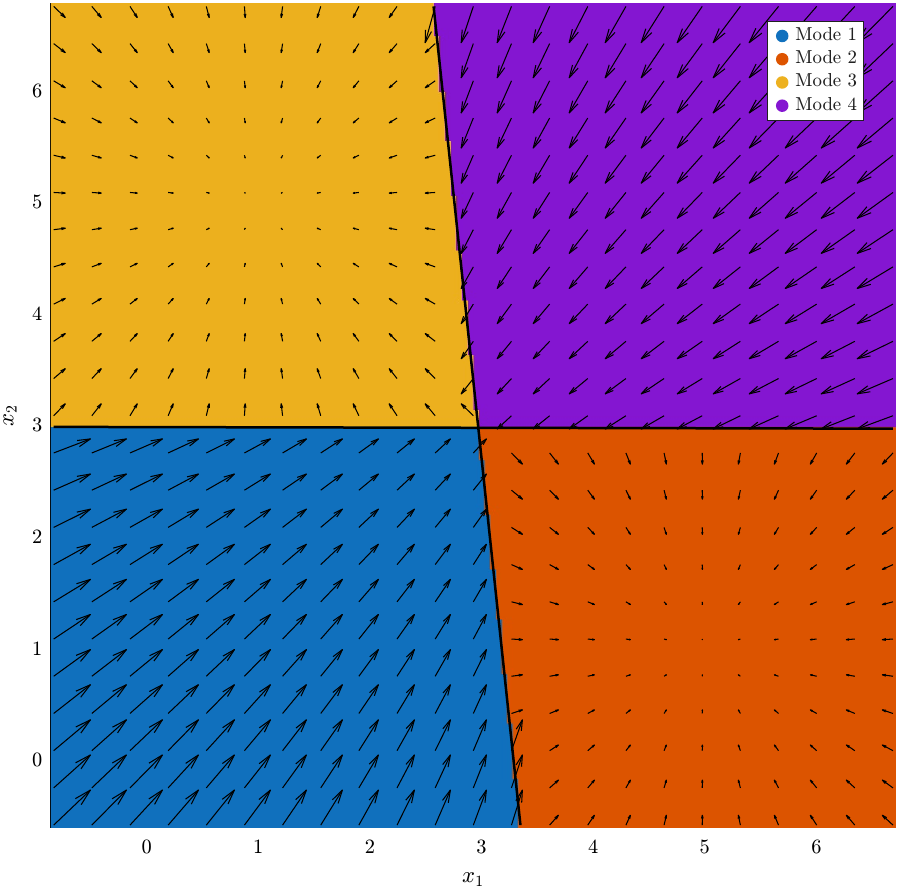}\\
    (a) Training samples & (b) Learned vector fields
    \end{tabular}
    \caption{Learned switching linear system colored by learned modes.}
    \label{fig:learned sls}
\end{figure}

\begin{figure}[ht]
  \centering
  \footnotesize
  \begin{tabular}{cccc}
    \multicolumn{4}{c}{\textbf{Morse Sets}} \\[0.3em]
    \includegraphics[width=0.22\textwidth]{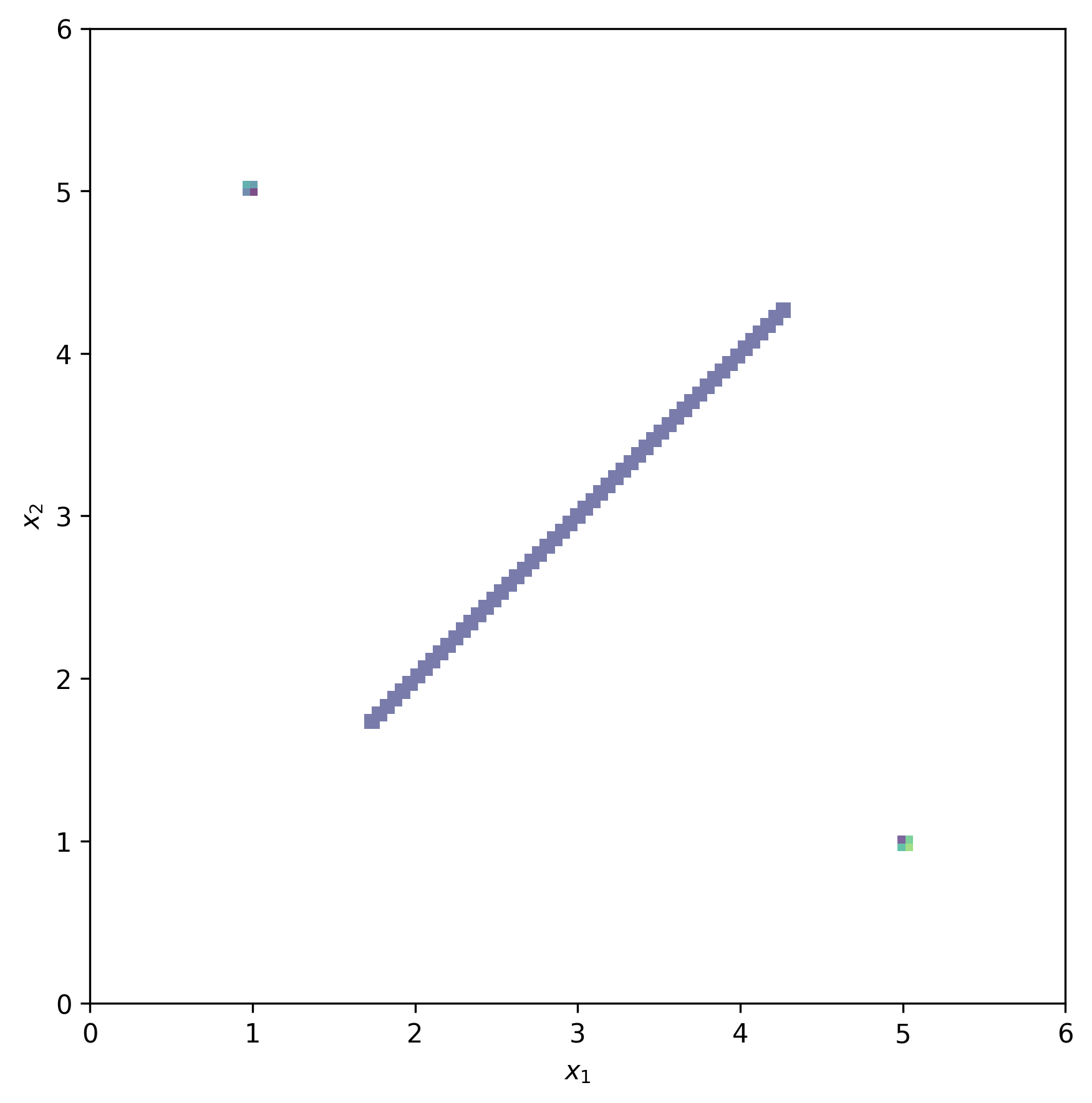} &
    \includegraphics[width=0.22\textwidth]{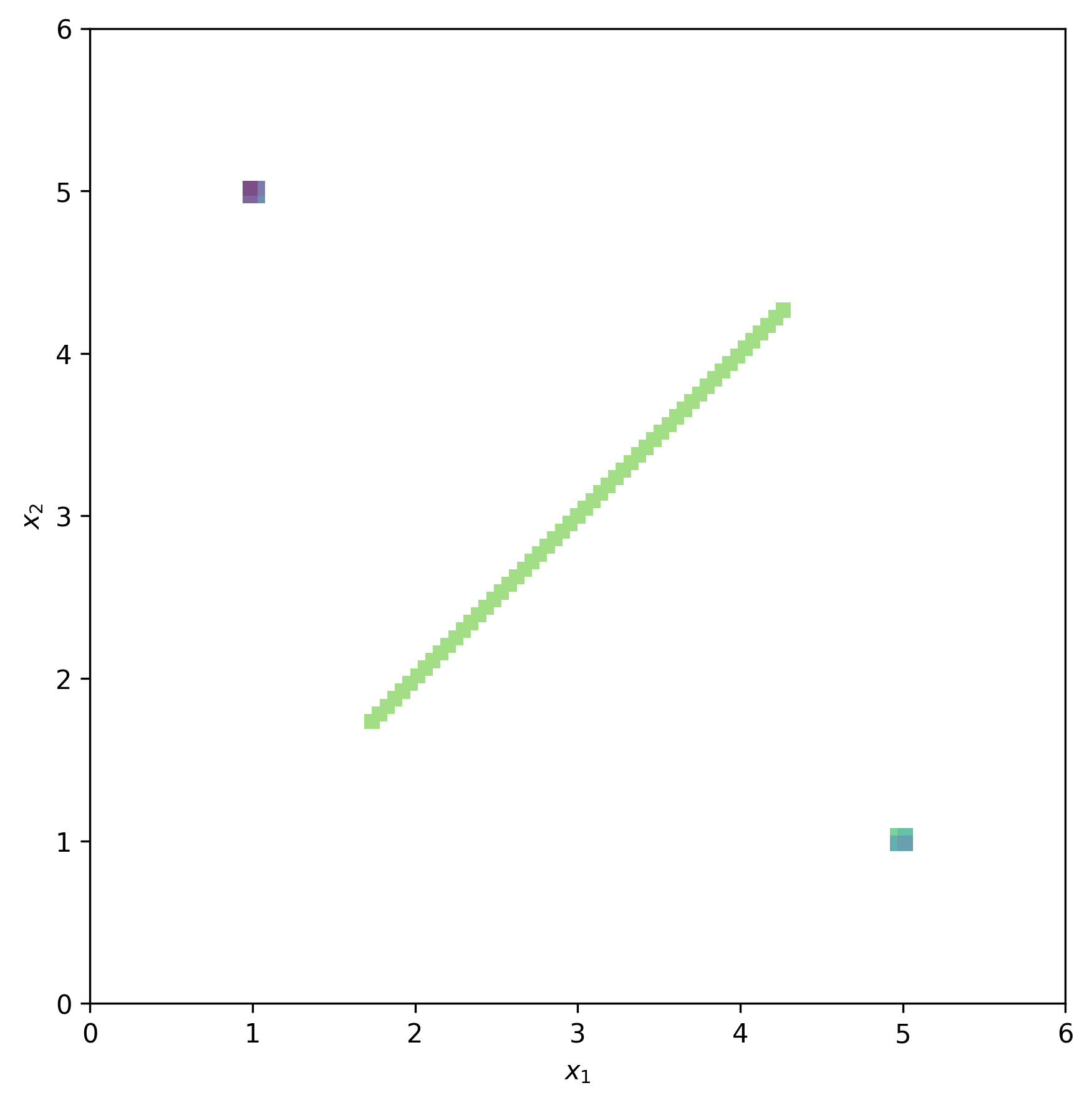} &
    \includegraphics[width=0.22\textwidth]{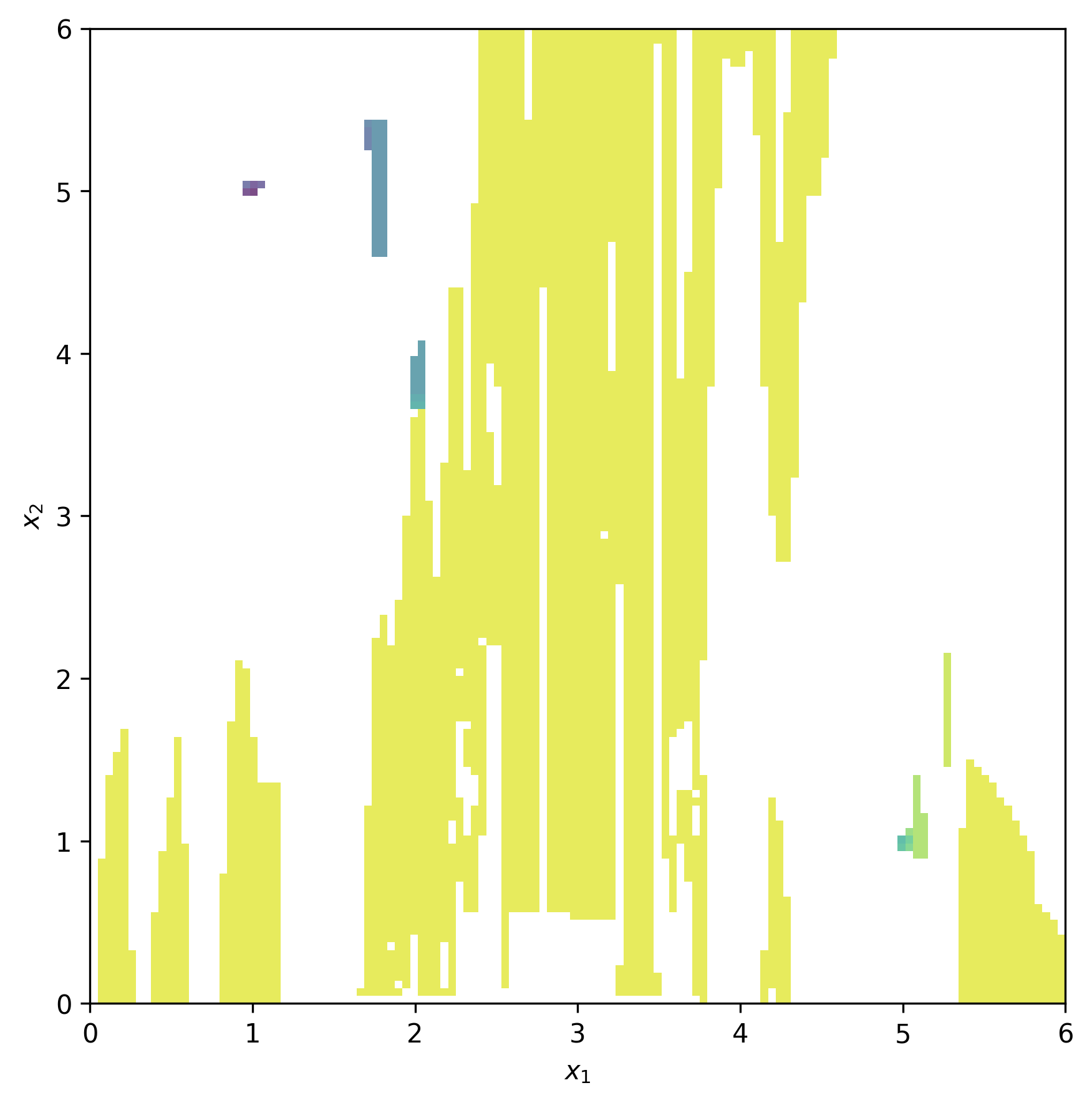} &
    \includegraphics[width=0.22\textwidth]{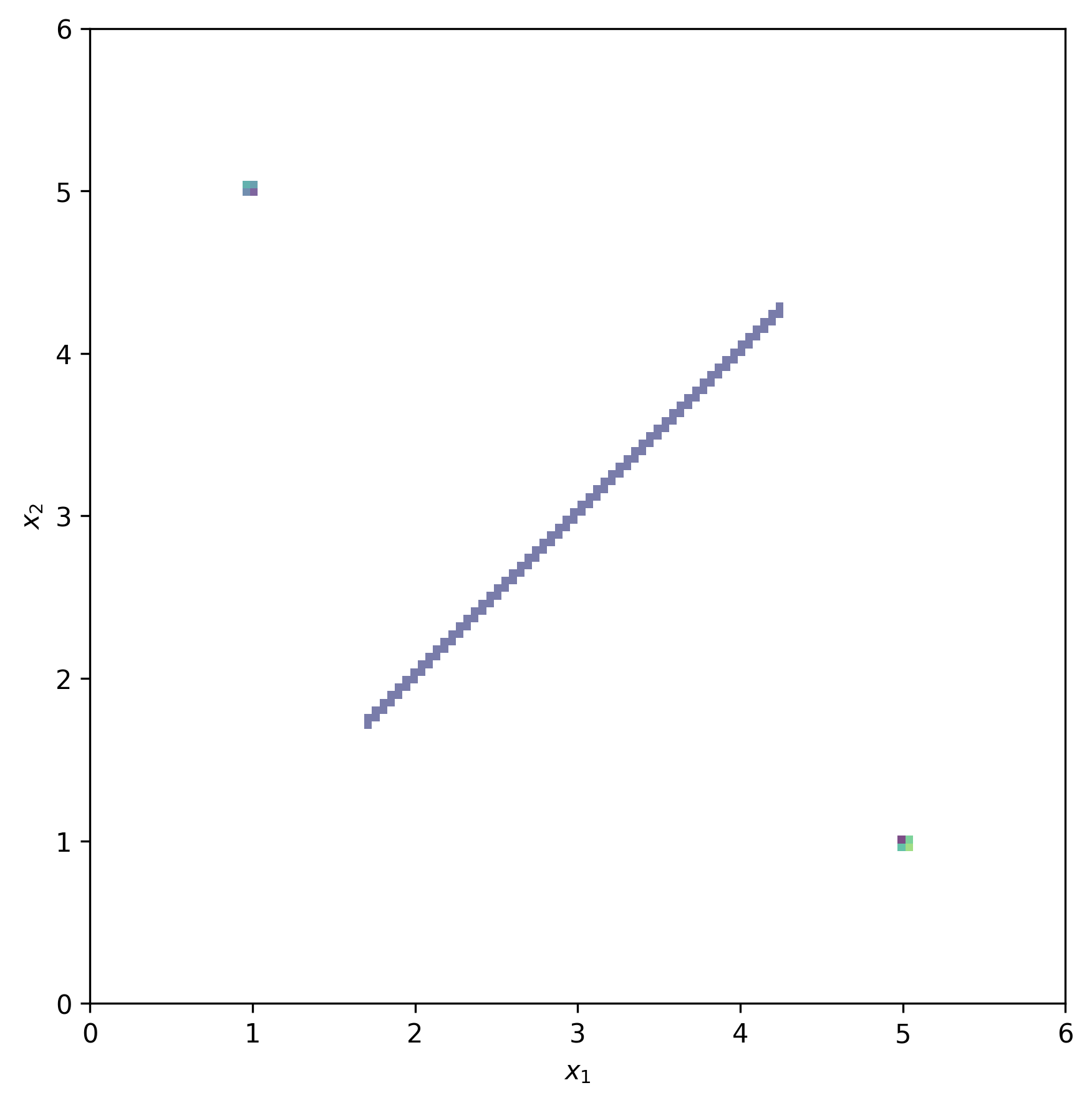} \\[0.2em]
    \multicolumn{4}{c}{\textbf{Morse Graphs}} \\[0.3em]
    \includegraphics[width=0.22\textwidth]{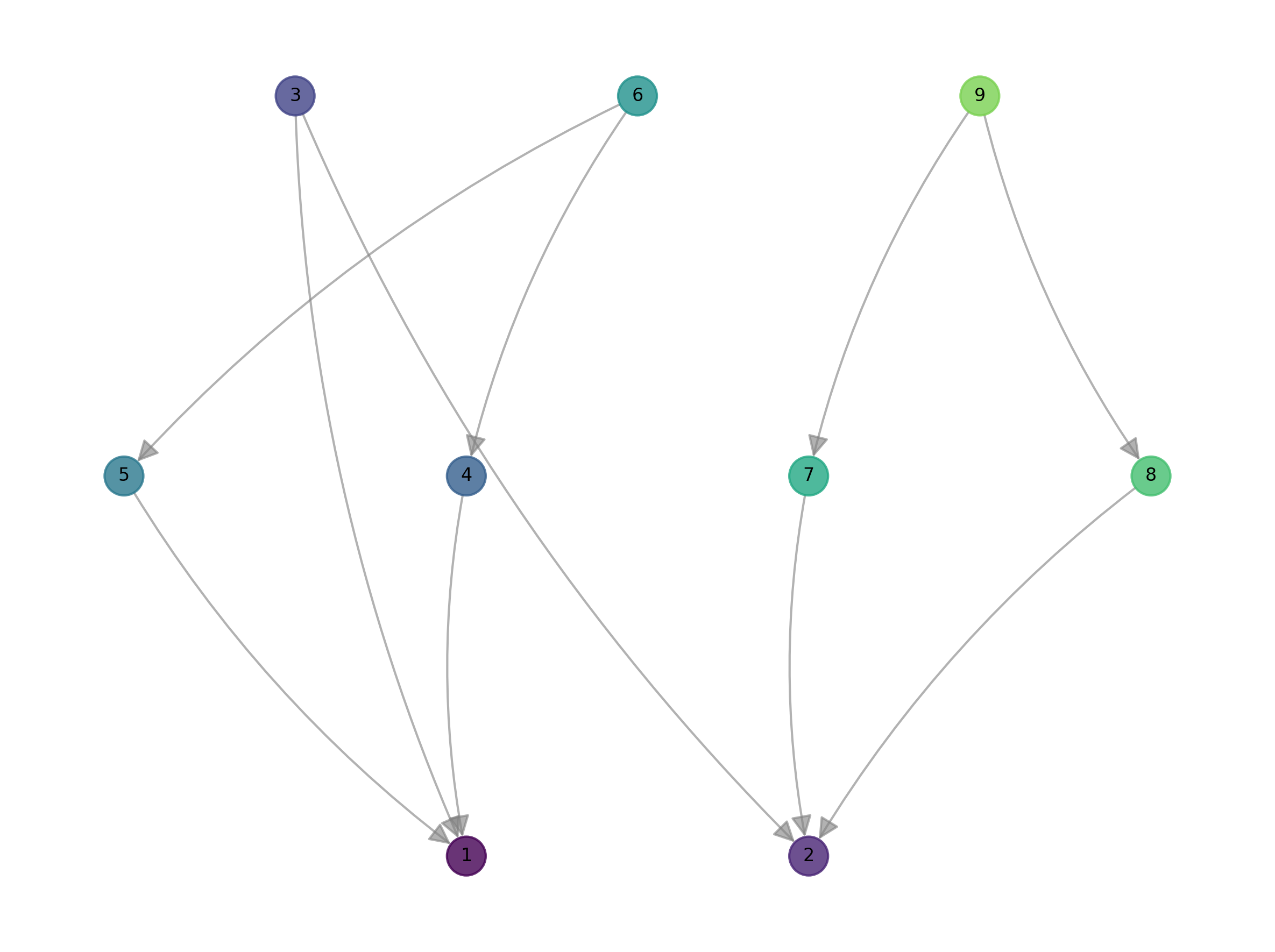} &
    \includegraphics[width=0.22\textwidth]{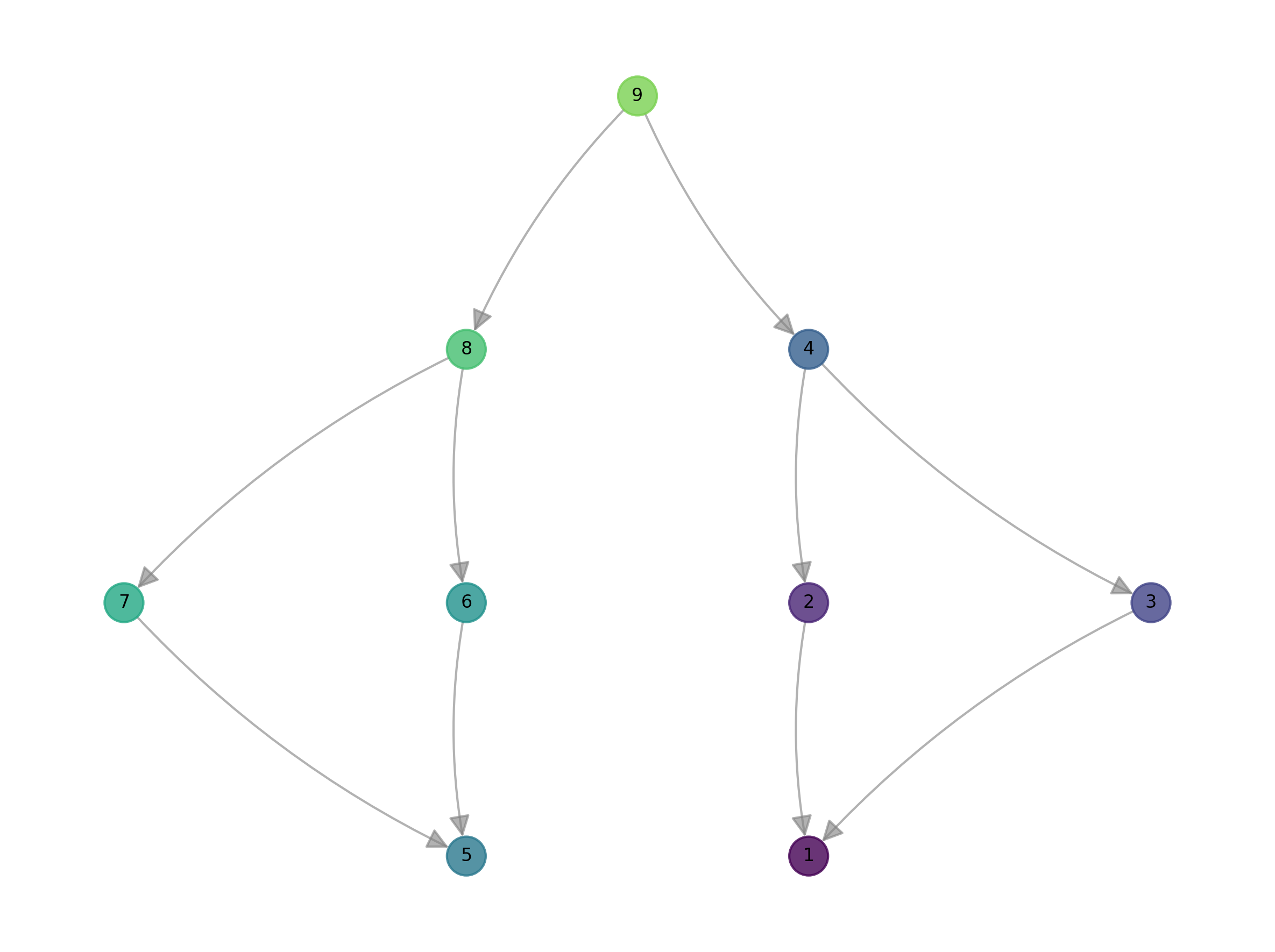} &
    \includegraphics[width=0.22\textwidth]{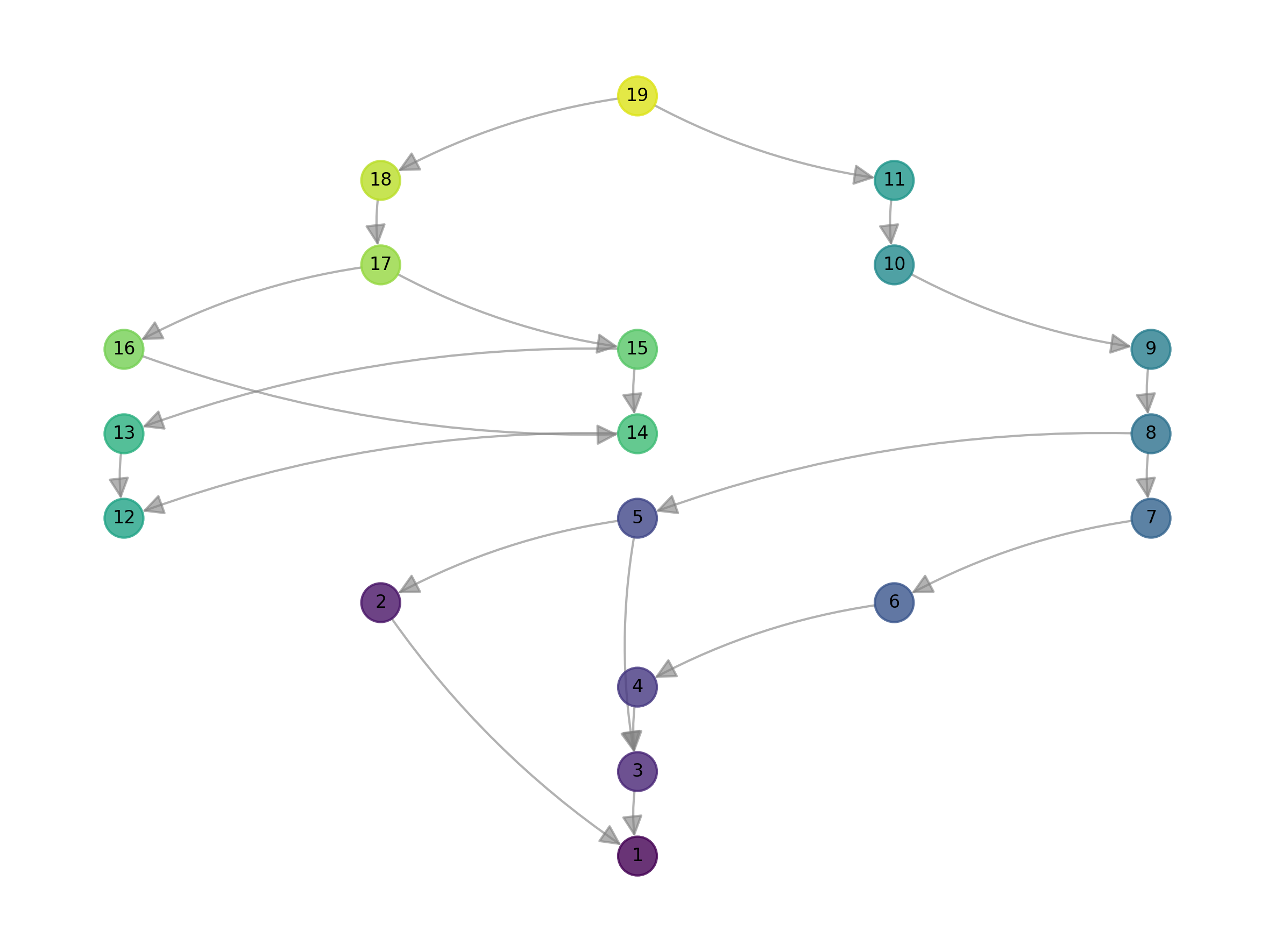} &
    \includegraphics[width=0.22\textwidth]{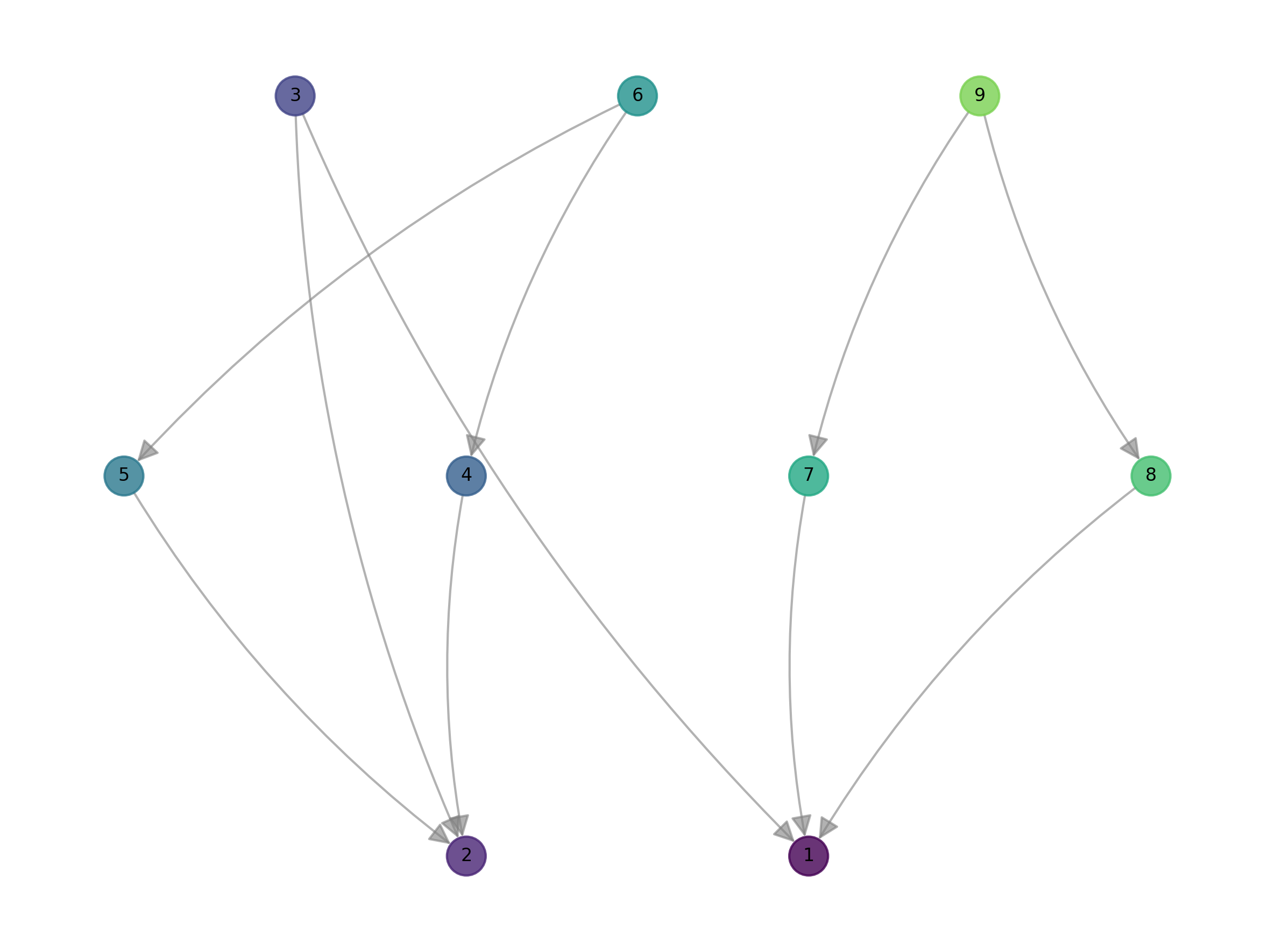} \\
    (a) Ground-Truth & (b) Lipschitz & (c) Gaussian-Process & (d) Ours
  \end{tabular}
  \caption{Morse sets (top) and Morse graphs (bottom).}
  \label{fig:toggleswitch}
\end{figure}

\begin{figure}[ht]
  \centering
  \footnotesize
  \begin{tabular}{ccccc}
    \raisebox{0.3\height}{
      \begin{tikzpicture}[scale=0.9,
        morse/.style={circle, draw, line width=1.2pt, minimum size=20pt, font=\bfseries}]
        \node[morse, fill=magenta!80, text=white] (n3) at (0,1.3) {3};
        \node[morse, fill=orange!40!yellow!60!white] (n1) at (-0.8,0) {1};
        \node[morse, fill=cyan!60!blue!90, text=white] (n2) at (0.8,0) {2};
        \draw[->, line width=1.2pt] (n3) -- (n1);
        \draw[->, line width=1.2pt] (n3) -- (n2);
      \end{tikzpicture}
    } &
    \includegraphics[width=0.17\textwidth]{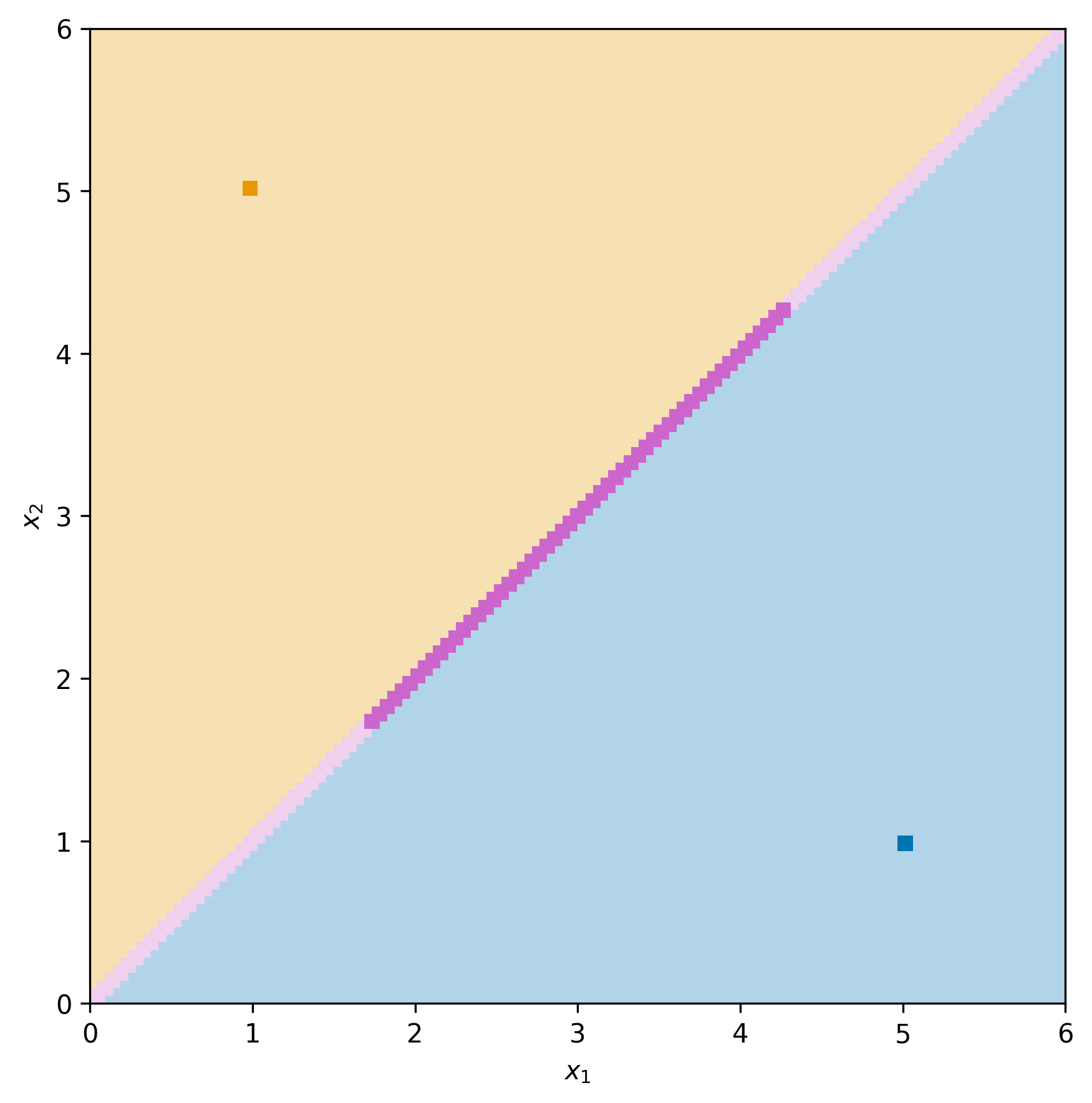} &
    \includegraphics[width=0.17\textwidth]{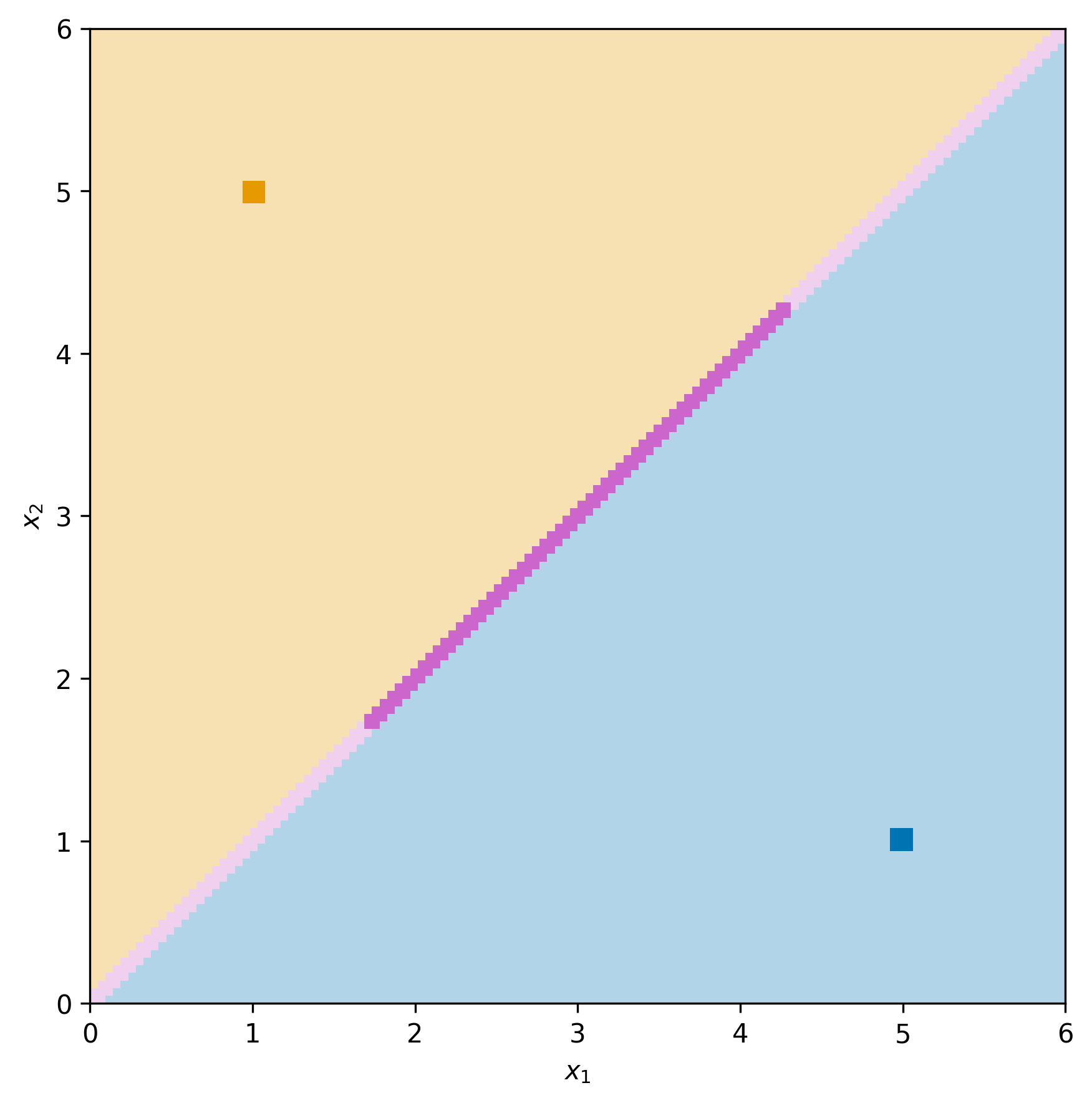} &
    \includegraphics[width=0.17\textwidth]{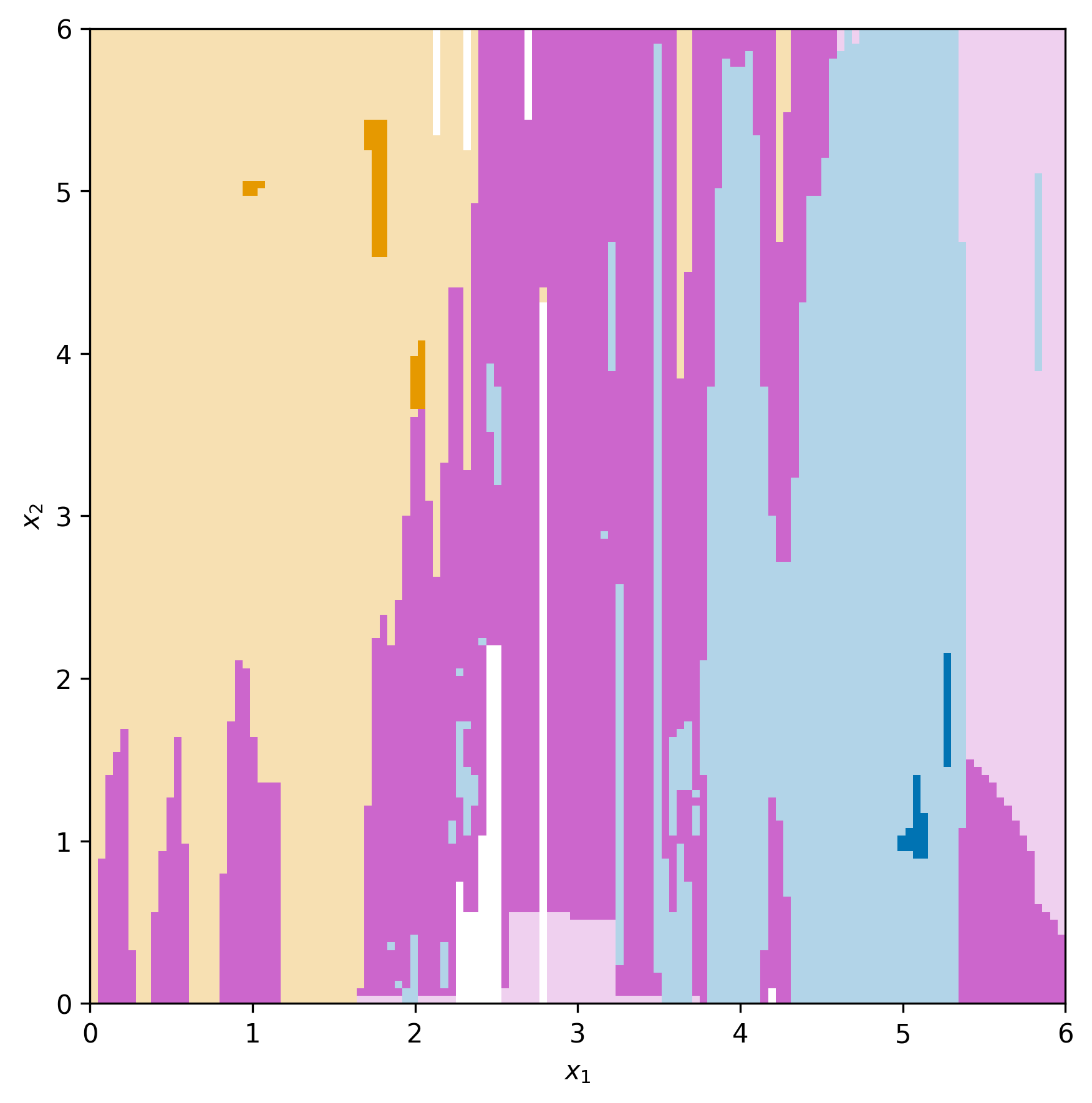} &
    \includegraphics[width=0.17\textwidth]{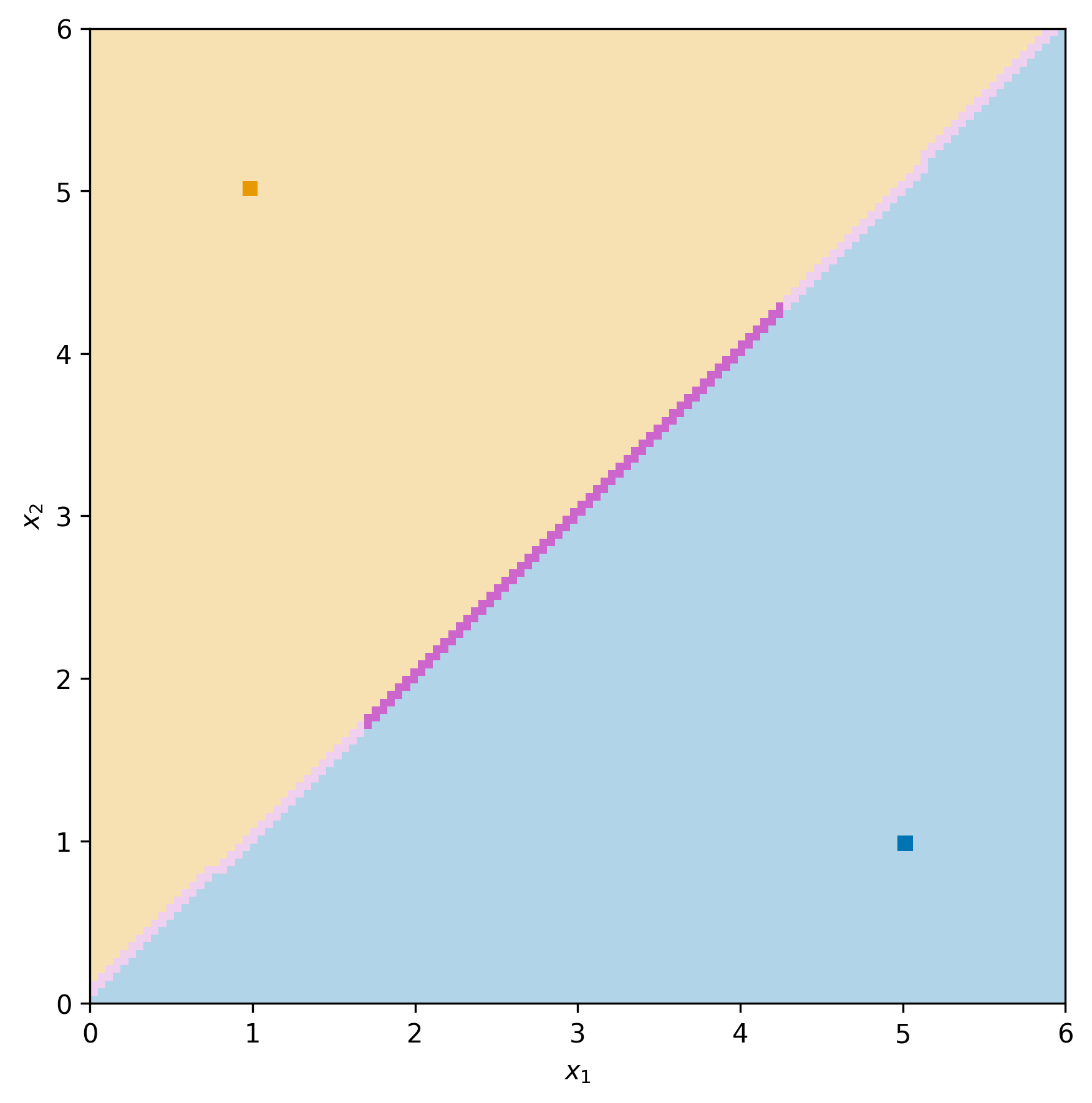} \\
    (a) Agg $\MG$ & (b) Ground-Truth & (c) Lipschitz & (d) Gaussian-Process & (e) Ours
  \end{tabular}
  \caption{Aggregated Morse graph (a) and geometric realization of Morse sets and regions of attraction.}
  \label{fig:toggle_combined}
\end{figure}

\begin{table}[ht]
  \centering
  \small
  \begin{tabular}{l|c|ccc|ccc}
    \hline
    Method & \# $\MG(\cF)$ & $M(1)$ & $M(2)$ & $M(3)$ & $\RoA(1)$ & $\RoA(2)$ & $\RoA(3)$ \\
    \hline
    Ground-Truth & 9 & 1.000 & 1.000 & 1.000 & 1.000 & 1.000 & 1.000 \\
    Lipschitz & 9 & 0.444 & 0.444 & \textbf{1.000} & \textbf{1.000} & \textbf{1.000} & \textbf{1.000} \\
    Gaussian-Process & 19 & 0.108 & 0.066 & 0.026 & 0.385 & 0.552 & 0.982 \\
    Ours & 9 & \textbf{1.000} & \textbf{1.000} & 0.663 & 0.980 & 0.996 & \textbf{1.000} \\
    \hline
  \end{tabular}
  
  \caption{Quantitative metrics for $\MG(\cF)$ and Aggregated Morse Sets and RoAs.}
  \label{tab:toggle_metrics}
\end{table}

The Lipschitz approach overestimates the recurrent structure associated to attractors compared to our method; however it reconstructs with high-fidelity the RoAs when compared with the reference method. This is expected given that the local estimate for the Lipschitz bound coincides with the global bound. The GP method produces 19 Morse sets, which differs significantly from the other computations. This fragmentation occurs because GPs assume smoothness while the system switches discontinuously at thresholds $T_1, T_2$. Standard kernels (RBF, Matérn) cannot capture switching behavior without excessive local variation, and confidence intervals expand near discontinuities, creating spurious sets. \citet{Ewerton2022GP} suggests a refinement step to sample from regions with large variation. Notably, the GP outer approximation performs better when $\tau >0$ is a smaller time step or the refinement step is performed. 

The topological dynamics reveals deeper structure. In terms of attracting blocks, we can see that $H_*(\RoA(3), \RoA(1) \cup \RoA(2)) \cong \tilde{H}_*(S^1)$ where $S^1$ is the circle. The index describes the algebraic topological information when collapsing the attracting block $\RoA(1) \cup \RoA(2)$, contained in $\RoA(3)$, to a single point, providing a homological description of the maximal invariant set contained in $M(3)$. In this case, it corresponds to the Conley index of a saddle-point in classical dynamics. Similarly, $CH_*(1)= CH_*(2)$ yield the Conley index of stable fixed points. Since the latter Morse sets lie entirely within regions of continuous dynamics, we may conclude they contain fixed points \citep{Handbook2002}.

\subsection{Piecewise Van der Pol: Limit Cycle}

We consider a Switching Polynomial System (SPS) inspired by the Van der Pol oscillator. Let $x = (x_1, x_2) \in \R^2$ and consider
\[
  \begin{aligned}
    \dot{x}_1 &= x_2 \\
    \dot{x}_2 &= (1-x_1^2)x_2 - g(x_1)
  \end{aligned} \quad \text{ where } \quad
  g(x_1) =
  \begin{cases}
    x_1^3 & |x_1| < 1 \\
    x_1 & |x_1| \geq 1
  \end{cases}
\]
The switching occurs at $|x_1| = 1$ when the force changes from cubic to linear. The system exhibits a stable limit cycle encircling an unstable equilibrium at the origin. Unlike the toggle switch, the vector field remains continuous at the switching boundary. This continuity and the nonlinear term fundamentally changes the performance of different approximation methods.

We generate $N = 50$ trajectories over $[0,10]$ with $\tau = 1.0$ and discretize $[-3,3]^2$ into $2^7 \times 2^7$ cubical cells. Figure~\ref{fig:vdp_combined} shows the aggregated results. Table~\ref{tab:vdp_metrics} summarizes the quantitative comparison. We learn $\hat{f}$ as a switching polynomial system with modes separated by $s(x) \approx -0.49x_1^2$.

\begin{figure}[ht]
    \centering
    \footnotesize
    \begin{tabular}{cc}
    \includegraphics[width=0.28\textwidth]{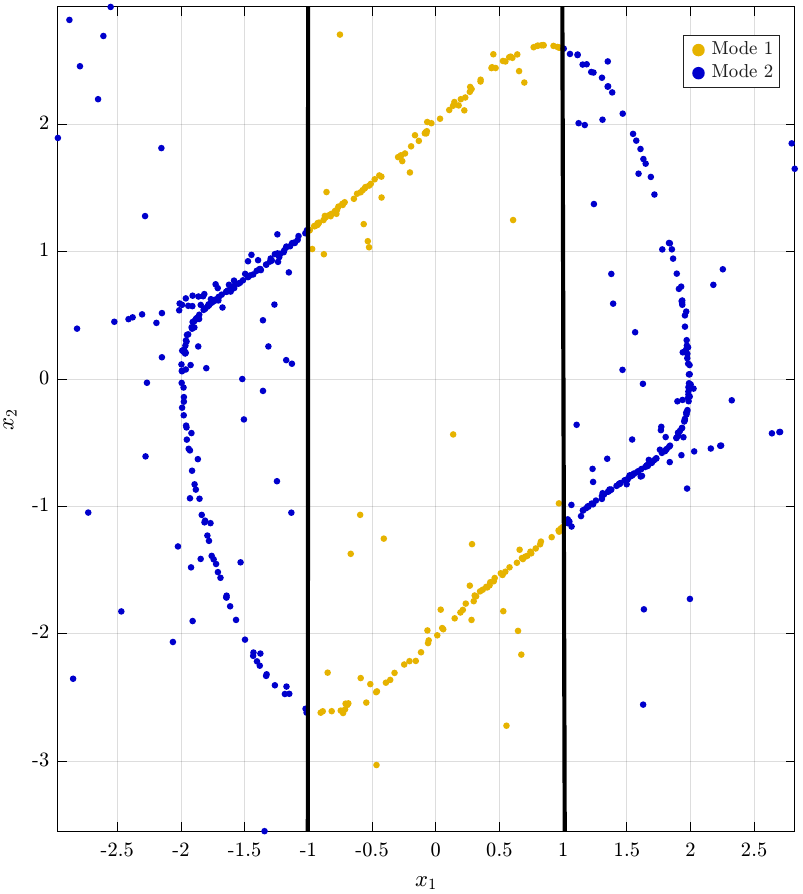}&
    \includegraphics[width=0.28\textwidth]{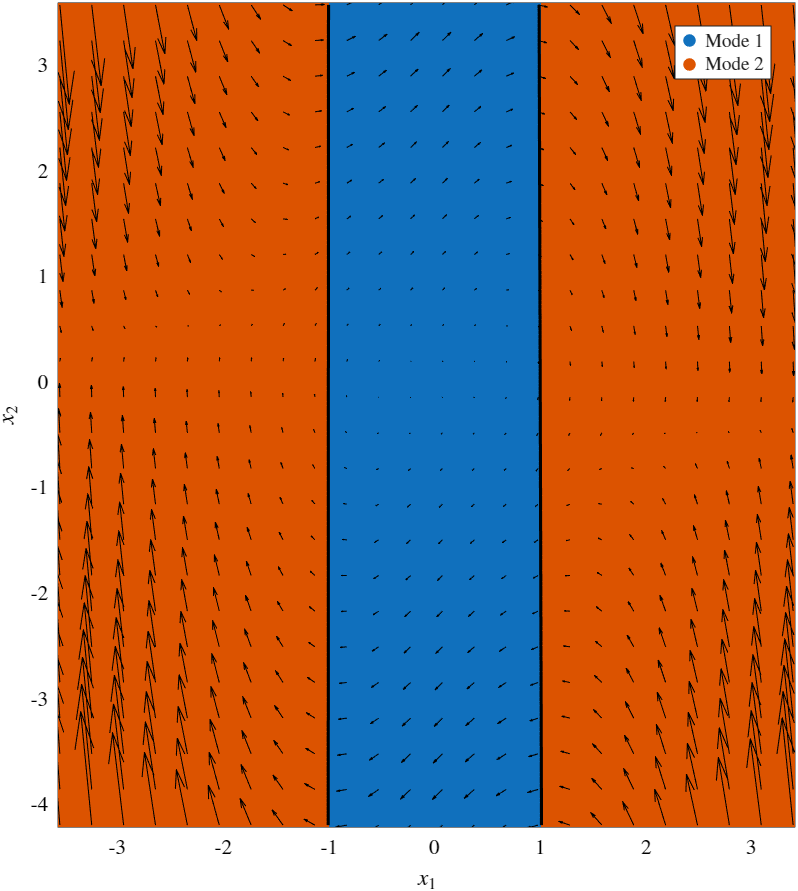}\\
    (a) Training samples & (b) Learned vector fields
    \end{tabular}
    \caption{Learned switching polynomial system colored by learned modes.}
    \label{fig:learned sps}
\end{figure}

\begin{figure}[ht]
  \centering
  \footnotesize
  \begin{tabular}{ccccc}
    \raisebox{0.3\height}{
      \begin{tikzpicture}[scale=0.9,
        morse/.style={circle, draw, line width=1.2pt, minimum size=20pt, font=\bfseries}]
        \node[morse, fill=orange!70!yellow!60!white] (n1) at (0,1.3) {1};
        \node[morse, fill=cyan!60!blue!90, text=white] (n0) at (0,0) {0};
        \draw[->, line width=1.2pt] (n1) -- (n0);
      \end{tikzpicture}
    } &
    \includegraphics[width=0.18\textwidth]{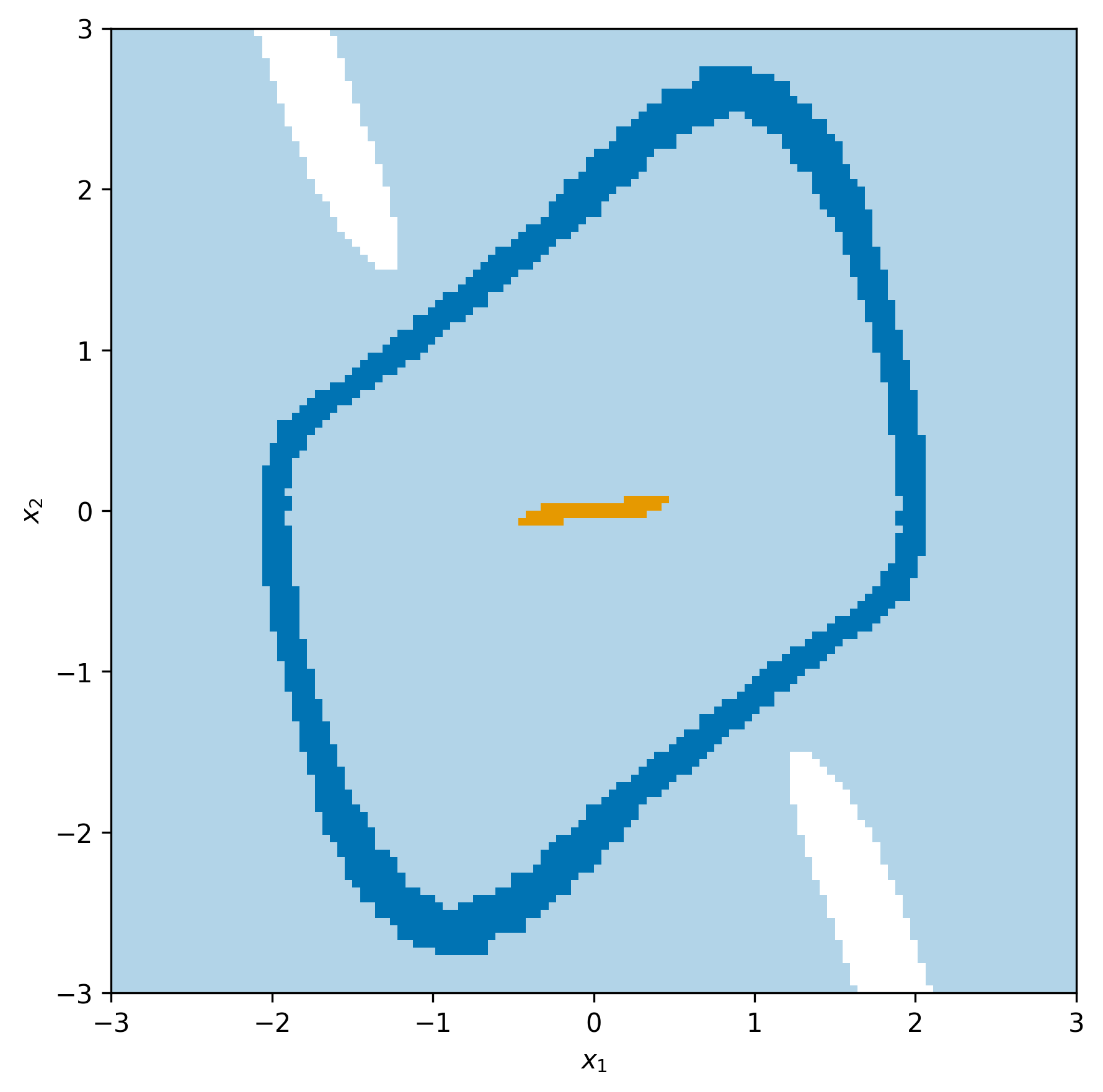} &
    \includegraphics[width=0.18\textwidth]{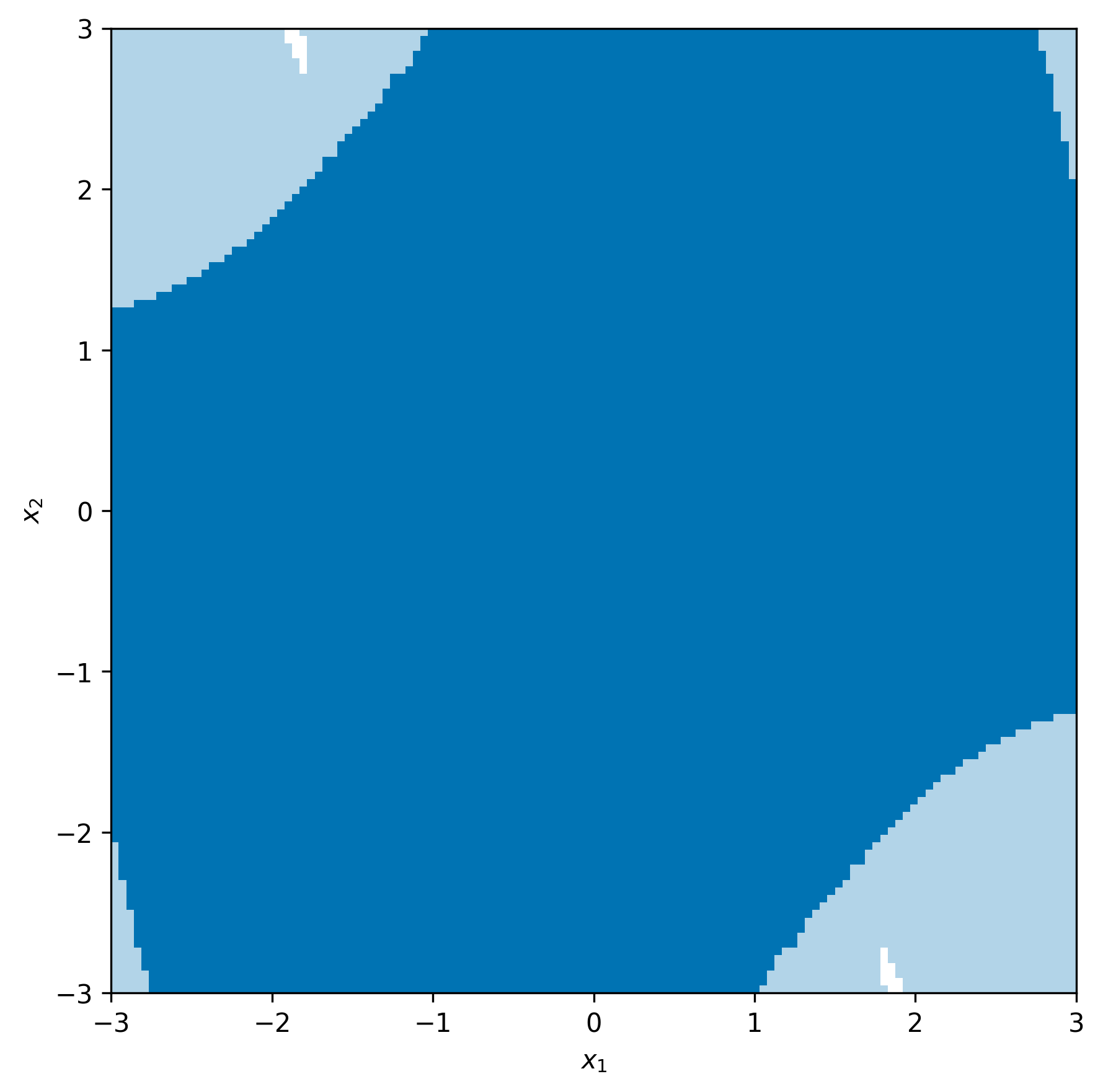} &
    \includegraphics[width=0.18\textwidth]{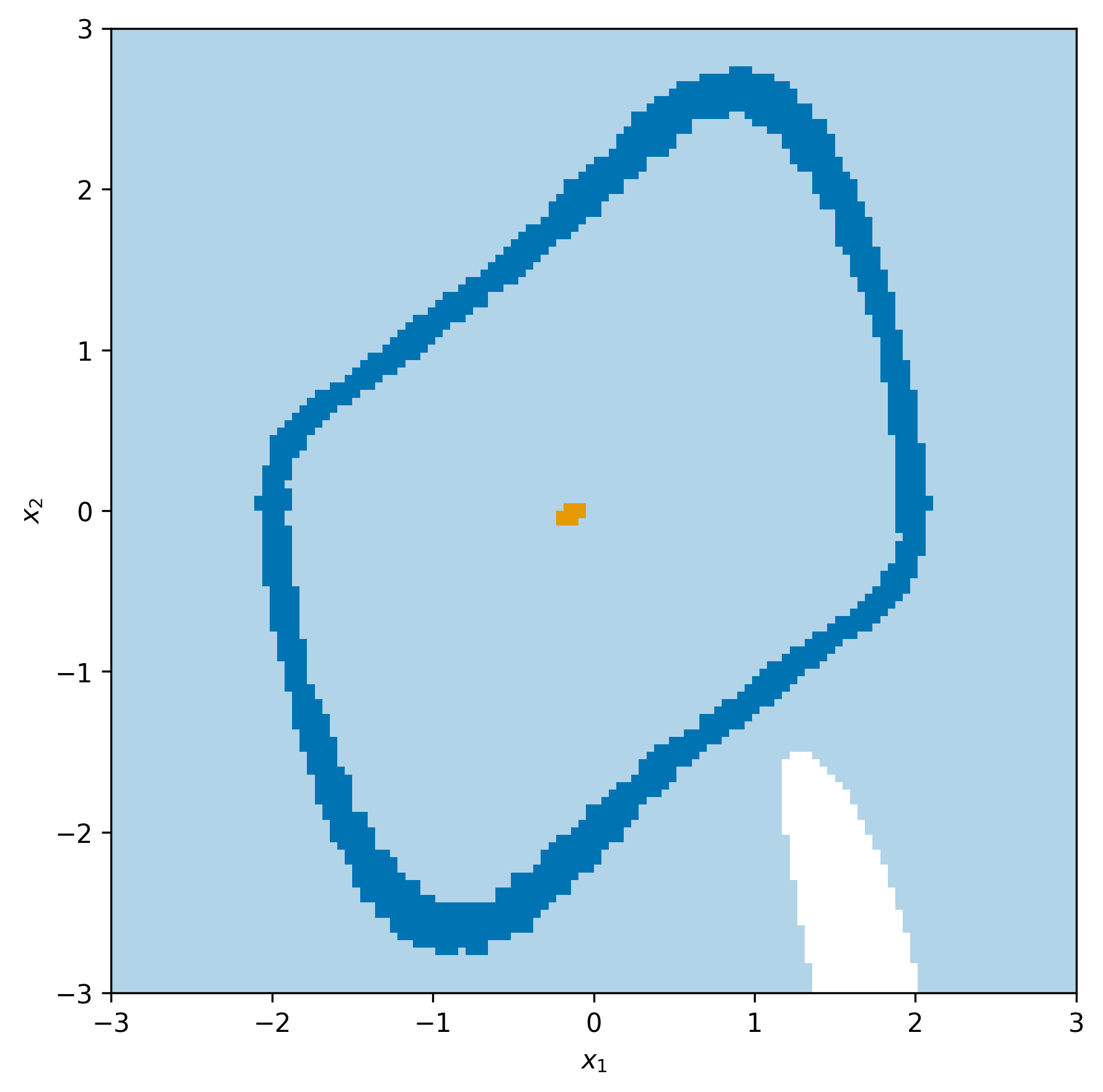} &
    \includegraphics[width=0.18\textwidth]{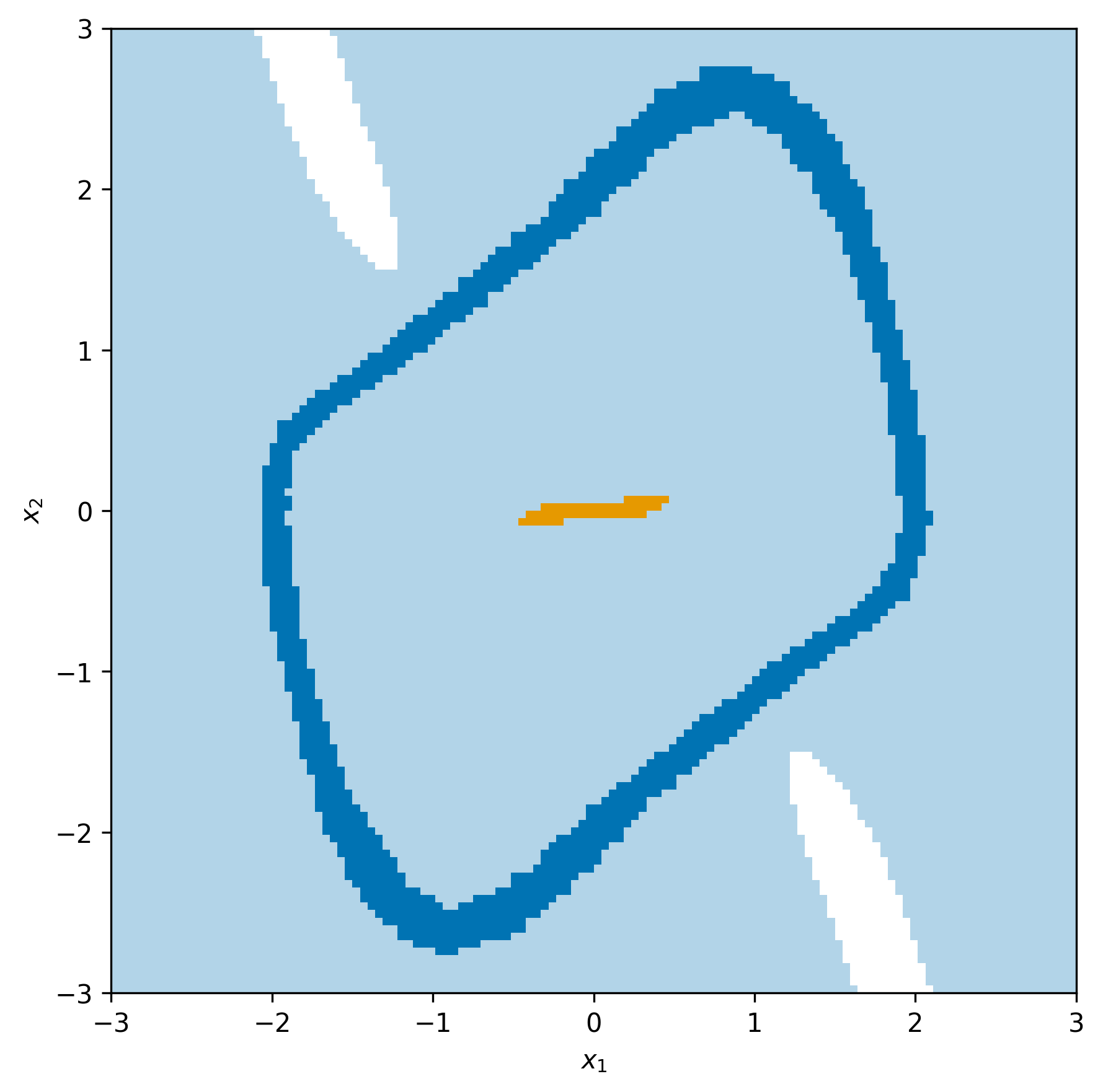} \\
    (a) Agg $\MG$ & (b) Ground-Truth & (c) Lipschitz & (d) Gaussian-Process & (e) Ours
  \end{tabular}
  \caption{Aggregated Morse graph and geometric realization of Morse sets and regions of attraction.}
  \label{fig:vdp_combined}
\end{figure}

\begin{table}[ht]
  \centering
  \small 
  \begin{tabular}{lcccc}
    \hline
    Method & \# $\MG(\cF)$ & $M(1)$ & $M(2)$ & $\RoA(1)$ \\
    \hline
    Ground-Truth & 4 & 1.000 & 1.000 & 1.000 \\
    Lipschitz & 1 & 0.094 & 0.000 & 0.964 \\
    Gaussian-Process & 2 & 0.932 & 0.174 & 0.973 \\
    Ours & 4 & \textbf{0.972} & \textbf{1.000} & \textbf{1.000} \\
    \hline
  \end{tabular}
    \caption{Quantitative metrics for $\MG(\cF)$ and Aggregated Morse Sets and RoAs.}
  \label{tab:vdp_metrics}
\end{table}

Our method identifies four Morse sets, the same number as the reference method. The GP method shows better results than the previous example since the continuity of the vector field allows standard kernels to approximate the flow more effectively. In contrast, the Lipschitz method obtains a single Morse set due to an overly conservative global bound. In other words, the nonlinearity causes the Lipschitz constant to be too large to discern the unstable behavior at the origin.

After aggregation, these correspond to two dynamical structures: an unstable fixed point (repeller) and a limit cycle (attractor). The Conley index in the planar case can be computed by inspection, and due to the existence of a Poincaré section, the index $CH_*(1)\cong H_*(S^1)$ leads to the existence of a periodic orbit in $M(1)$ \citep{McCord1995}. By collapsing $\RoA(1)$ to a point, one obtains the Conley index of an unstable fixed point in $M(2)$, namely, $CH(2)\cong \tilde{H}_*(S^2)$ where $S^2$ is the sphere.

\section{Conclusion}

We have presented a framework for topological dynamics by integrating switching system identification via convex optimization with combinatorial dynamics. The identified analytical model enables explicit construction of outer approximations, addressing a key limitation of data-driven approaches.

Our experimental results on planar systems show that the method recovers attractors and regions of attraction, achieving better results than Gaussian-Process and Lipschitz methods used in robotics applications \citep{Ewerton2022GP,Ewerton2022MG}. We restricted our experiments to low-dimensional systems as our prototype implementation prioritized interpretability over computational efficiency. However, faster methods to compute cubical homology and combinatorial dynamics could be used \citep{CHomP, CMGDB}.

Future work includes the extension to higher-dimensional systems through adaptive refinement and sparse grid methods, computation of the Conley index of the $\tau$-map via shift-equivalences, and exploration of hybrid systems in the control setting by identifying control modes.

\acks{Supported in part by AFOSR grant FA9550-23-1-0400 (MURI),  AFOSR
grant FA9550-32-1-0215, AFOSR grant FA9550-23-1-0011, NSF grant 2103026.}

\bibliography{l4dc2026-sample}

\end{document}